\documentclass{llncs}
\usepackage{graphicx}
\usepackage{enumerate,enumitem}
\usepackage{verbatim}
\usepackage{amsmath}
\usepackage{amsfonts}
\usepackage{url}
\usepackage{multirow}
\usepackage{simplemargins}
\setallmargins{1in}

\DeclareGraphicsExtensions{.pdf,.png,.eps}
\graphicspath{{figures/}}

\let\doendproof\endproof
\renewcommand\endproof{~\hfill\qed\doendproof}

{\bfseries}{\itshape}

\newcommand{\LL}{\mathcal{L}}
\newcommand{\I}{\oplus\oplus}
\newcommand{\II}{\ominus\oplus}
\newcommand{\III}{\ominus\ominus}
\newcommand{\IV}{\oplus\ominus}

\begin{document}
\title{Combinatorial and Geometric Properties of Planar Laman Graphs}
%\thanks{This work was initiated at the Bertinoro Workshop on Graph Drawing in 2012.}

\author{
Stephen Kobourov\inst{1}\thanks{Research supported in part by NSF grants CCF-0545743 and CCF-1115971.}
\and Torsten Ueckerdt\inst{2}\thanks{Research was supported by GraDR EUROGIGA project No. GIG/11/E023.}
\and Kevin Verbeek\inst{3}\thanks{Research was supported by the Netherlands Organisation for Scientific Research (NWO) under project no.~639.022.707.}
}

\institute{
Department of Computer Science,
University of Arizona
\and
Department of Applied Mathematics, Charles University
\and
Department of Mathematics and Computer Science, TU Eindhoven\\
{\tt kobourov@cs.arizona.edu \qquad torsten@kam.mff.cuni.cz \qquad k.a.b.verbeek@tue.nl}
}
%\thanks{ Research supported in part by: NSERC}
\maketitle
\pagestyle{plain}

\begin{abstract}
%We introduce three new combinatorial structures for planar Laman graphs: angular structures, angle labelings, and edge labelings. The latter two structures are related to Schnyder realizers for maximally planar graphs. We prove that planar Laman graphs are exactly the class of graphs that
%have an angular structure that is a tree, called angular tree. Angular trees have a corresponding angle labeling and edge labeling. Using a combination of these powerful combinatorial structures, we show that every planar Laman graph $G$ has an $L$-contact representation. This $L$-contact representation can be drawn on an $n\times n$ grid and can be computed in quadratic time in $n$, where $n$ is the number of vertices of $G$.
Laman graphs naturally arise in structural mechanics and rigidity theory. Specifically, they characterize minimally rigid planar bar-and-joint systems which are frequently needed in robotics, as well as in molecular chemistry and polymer physics.
We introduce three new combinatorial structures for planar Laman graphs: angular structures, angle labelings, and edge labelings. The latter two structures are related to Schnyder realizers for maximally planar graphs. We prove that planar Laman graphs are exactly the class of graphs that have an angular structure that is a tree, called \emph{angular tree}, and that every angular tree has a corresponding angle labeling and edge labeling.
\medskip

Using a combination of these powerful combinatorial structures, we show that every planar Laman graph has an L-contact representation, that is, planar Laman graphs are contact graphs of axis-aligned L-shapes.
%, where we allow all four rotations.
Moreover, we show that planar Laman graphs and their subgraphs are the only graphs that can be represented this way.
\medskip

We present efficient algorithms that compute, for every planar Laman graph $G$, an angular tree, angle labeling, edge labeling, and finally an L-contact representation of $G$. The overall running time is $\mathcal{O}(n^2)$, where $n$ is the number of vertices of $G$, and the L-contact representation is realized on the $n\times n$~grid.
\end{abstract}

\newpage

\setcounter{page}{1}

%%%%%%%%%%%%%%%%%%%%%%%%%%%%%%%%%%%%%%%
%%%%%%% I N T R O D U C T I O N %%%%%%%
%%%%%%%%%%%%%%%%%%%%%%%%%%%%%%%%%%%%%%%
% 2-segment graphs
% Connection to Schnyder
% application: L-contact graphs
% related work
%%%%%%%%%%%%%%%%%%%%%%%%%%%%%%%%%%%%%%%
\section{Introduction}\label{sec:introduction}
A \emph{contact graph} is a graph whose vertices are represented by geometric objects (like curves, line segments, or polygons), and edges correspond to two objects touching in some specified fashion. There is a large body of work about representing planar graphs as contact graphs. An early result is Koebe's 1936 theorem~\cite{Koebe36} that all planar graphs can be represented by
touching disks.

%There is a large body of work about representing planar graphs as
%contact graphs, i.e., graphs whose vertices are represented by
%geometric objects with edges corresponding to two objects touching
%in some specified fashion. Typical classes of objects might be curves,
%line segments, or polygons. An early result is Koebe's 1936
%theorem~\cite{Koebe36} that all planar graphs can be represented by
% touching disks.

In the late 1990's Schnyder showed that maximally planar graphs contain rich combinatorial structure~\cite{s-epgg-90}. With an angle labeling and a corresponding edge labeling, Schnyder shows that maximally planar graphs can be decomposed into three edge disjoint spanning trees. This combinatorial structure can be transformed into a geometric structure to produce a straight-line crossing-free planar drawing of the graph with vertex coordinates on the integer grid. Later, de~Fraysseix {\em et al.}~\cite{FraysseixTContact} show how to use the combinatorial structure to produce a representation of planar graphs as $T$-contact graphs (vertices are axis-aligned $T$'s and edges correspond to point contact between $T$'s) and triangle contact~graphs.

We study the class of {\em planar Laman graphs} and show that we can find similarly powerful combinatorial structures. In particular, we show that every planar Laman graph $G$ contains an \emph{angular structure}---a graph on the vertices and faces of $G$ with certain degree restrictions---that is also a tree and hence called an \emph{angular tree}. We also show that every angular tree has a corresponding \emph{angle labeling} and \emph{edge labeling}, which can be thought of as a special Schnyder realizer~\cite{s-epgg-90}. Using a combination of these combinatorial structures we show that planar Laman graphs are \emph{L-contact graphs}, graphs that can be represented as the contacts of axis-aligned non-degenerate L's (where the vertices correspond to the L's and the edges correspond to non-degenerate point contacts between the corresponding~L's). As a by-product of our approach we obtain a new characterization of planar Laman graphs: a planar graph is a Laman graph if and only if it admits an angular tree. The L-contact representation can be computed in $\mathcal{O}(n^2)$ time and realized on the $n\times n$~grid, where $n$ is the number of vertices of $G$.

%While in general Laman graphs are not planar (e.g., $K_{3,3}$), the planar Laman graphs are important as they contain several classes of planar graphs of interest: series-parallel graphs, outer-planar graphs, and planar 2-trees.

% It is known that
%planar Laman graphs are segment contact graphs, where the slope of %the segments is arbitrary.

%Here we show that planar Laman graphs are $L$-contact graphs.

%We consider the contact representations of graphs with vertices
%represented by non-degenerate axis-aligned $L$'s and  adjacencies %represented
%by non-degenerate point contact between the corresponding $L$'s.

\smallskip
\noindent{\bfseries Related Work.} Koebe's theorem~\cite{Koebe36} is an early example of point-contact
representation and shows that a planar graph can be represented by
touching disks. Any planar graph also has a contact representation
where all the vertices are represented by triangles in 2D~\cite{FraysseixTContact}, or even cubes in 3D~\cite{Felsner11}.
%, as shown by de Fraysseix {\em et al.}~\cite{FraysseixTContact}.

Planar bipartite graphs can be represented by axis-aligned segment
contacts~\cite{CzyzowiczKU98,fop-rpgs-91,rt-rplbopg-86}.
% The same result for a subclass of those graphs appears %in~\cite{rt-rplbopg-86}.
Triangle-free planar graphs can be represented via
contacts of segments with only three slopes~\cite{CastroCDMN02}. Furthermore, every $4$-connected $3$-colorable planar graph and every $4$-colored planar graph without an induced $C_4$ using four colors can be represented as the contact graph of segments~\cite{fo-rcis-07}. More generally, planar Laman graphs can be represented with contacts of segments with arbitrary number of slopes and every contact graph of segments is a subgraph of a planar Laman graph~\cite{A+11}.

The class of planar Laman graphs is of interest due to the fact that
it contains several large classes of planar graphs (e.g.,
series-parallel graphs, outer-planar graphs, planar 2-trees). Laman graphs are also of interest in structural mechanics, robotics, chemistry and physics, due to their connection to rigidity theory, which dates back to the 1970's~\cite{Laman}.  A system of fixed-length bars and flexible joints connecting them is
minimally rigid if it becomes flexible once any bar is removed; planar Laman graphs correspond to rigid planar bar-and-joint systems~\cite{gss-cr-93,hors+-pmrgpt-05}.

While Schnyder realizers were defined for maximally planar graphs~\cite{schnyderposet89,s-epgg-90}, the notion generalizes to $3$-connected planar graphs~\cite{f-lspg-04}. Fusy's transversal structures~\cite{Fusy09} for irreducible triangulations of the 4-gon also provide combinatorial structure that can be used to obtain geometric results.
Both concepts are closely related to certain angle labelings. Angle labelings of quadrangulations and plane Laman graphs have been considered before~\cite{FelsnerHKO10}. However, for planar Laman graphs the labeling does not have the desired Schnyder-like properties.
In contrast, the labelings presented in this paper do have these properties.
%(ours has!!)

\smallskip
\noindent{\bfseries Results and Organization.} In Section~\ref{sec:structures} we introduce three combinatorial structures for planar Laman graphs. We first show that planar Laman graphs admit an angular tree. Next, we use this angular tree to obtain a corresponding angle labeling and edge labeling.
%, which are closely related to Schnyder realizers.
In Section~\ref{sec:L-contact} we use a combination of these combinatorial structures to show that planar Laman graphs are L-contact graphs. We then describe an algorithm to compute the L-contact representation of a planar Laman graph $G$ in $\mathcal{O}(n^2)$ time on the $n\times n$ grid. The running time of our algorithm is dominated by the computation of an angular tree of $G$. Given an angular tree, the algorithm runs in $\mathcal{O}(n)$ time. We also provide a detailed example illustrating the constructive algorithm.
%Appendix~\ref{app:example} contains a detailed example illustrating the constructive algorithm. Proofs omitted due to space restrictions can be found in Appendix~\ref{app:omit_proofs}.

%We show that for planar Laman graphs we can construct an edge labeling with corresponding angle labeling, which is a special case of a Schnyder realizer. Using this powerful combinatorial structure, we show that planar Laman graphs are $L$-contact graphs.

\pagebreak

%%%%%%%%%%%%%%%%%%%%%%%%%%%%%%%%%%%%%%%
%%%%%% P R E L I M I N A R I E S %%%%%%
%%%%%%%%%%%%%%%%%%%%%%%%%%%%%%%%%%%%%%%
% 2-segment graphs definition
% plane Laman Graphs
% edge-condition for 2-connected subgraphs
% Henneberg sequence
%%%%%%%%%%%%%%%%%%%%%%%%%%%%%%%%%%%%%%%
%\section{Preliminaries}\label{sec:prelim}
%
%Consider a planar graph that has a contact representation using line
%segments such that the intersection of any three segments is
%empty. These are known as 2-segment representations and the graphs as
%2-segment graphs~\cite{A+11}.
%
%A planar graph is a 2-segment graph if and only if for any $W\subseteq V$,
%we have $|E(W)|\leq 2|W|-3$.
%The necessity of the condition is easily seen.
% Let ${\cal S}$ be the set of segments of
%a 2-segment representation of $G$. For $W\subset V$ let $X_W$ be the set of
%end-points of segments in $\cal S$ corresponding to vertices of $W$.
%Since we have a 2-segment representation we may assume that $|X_W|=2|W|$.
%There is an injection $\phi$ from edges in $E(W)$ to points in $X_W$,
%points belonging to the convex hull of $X_W$, however, can not be in the
%image of $\phi$. Since the convex hull contains at least three points
%we get: $|E(W)| \leq |X_W| - 3 = 2|W|-3$.
%%
%\begin{definition}
% A Laman graph is a graph $G=(V,E)$ with $|E| = 2 |V| - 3$ and $|E(W)| \leq 2|W| - 3$ for all $W \subset V$.
%\end{definition}
%%
%%Laman graphs are of interest in rigidity-theory, see %e.g.~\cite{gss-cr-93,f-gga-04}.

%%%%%%%%%%%%%%%%%%%%%%%%%%%%%%%%%%%%%%%
%%%%%% COMBINATORIAL STRUCTURES %%%%%%%
%%%%%%%%%%%%%%%%%%%%%%%%%%%%%%%%%%%%%%%
\section{Combinatorial Structures for Planar Laman Graphs}\label{sec:structures}

%In this section we show how the combinatorial structure of the plane Laman graph $G$ and its dual (the angular graph) can be used to label the angles of $G$, which in turn can be used to label the edges of $G$. The resulting edge labeling is a special case of a Schnyder realizer, which we use to obtain geometric structure (L-contact representation) for plane Laman graphs (Section~\ref{sec:L-contact}).

Consider a graph $G = (V, E)$. For a subset of vertices $W \subseteq V$, let $G(W)$ be the subgraph of $G$ induced by $W$, and let $E(W)$ be the set of edges of $G(W)$.
%Let $E(W)$ be the edges of the subgraph $G(W)$ of $G=(V,E)$ induced by $W \subseteq V$.
%
\begin{definition}
A Laman graph is a connected graph $G=(V,E)$ with $|E|=2|V|-3$ and $|E(W)| \leq 2|W|-3$ for all $W \subset V$.
\end{definition}
%
%Equivalently, $G$ is a Laman graph if $|E(G)| = 2|V(G)| -3$ and for each $2$-connected induced subgraph $G(W)$ of $G$ we have $E(W) \leq 2|W| - 3$. Moreover, Laman graphs are $2$-connected. %Both facts are well-known, however we provide a proof in the appendix (Lemma~\ref{lem:2-connected-laman}).
%
%
Laman graphs admit a \emph{Henneberg construction}: an ordering $v_1,\ldots,v_n$ of the vertices such that, if $G_i$ is the graph induced $v_1,\ldots,v_i$, then $G_3$ is a triangle and $G_i$ is obtained from $G_{i-1}$ by one of the following~operations:
\begin{description}
\item{(${\bf H_1}$)} Choose two vertices $x$, $y$ from $G_{i-1}$ and add $v_i$ together with the edges $(v_i,x)$ and $(v_i,y)$.
\item{(${\bf H_2}$)} Choose an edge $(x,y)$ and a third vertex $z$ from $G_{i-1}$, remove $(x,y)$ and add $v_i$ together with the three edges $(v_i, x)$, $(v_i, y)$, and $(v_i, z)$.
\end{description}
Planar Laman graphs also admit a {\em planar Henneberg construction}~\cite{hors+-pmrgpt-05}. That is, the graph can be constructed together with a plane straight-line embedding, with each vertex remaining in the position it is inserted. The two operations of a (planar) Henneberg construction are illustrated in Figure~\ref{fig:Henneberg}.
\begin{figure}[h]
  \centering
  \includegraphics{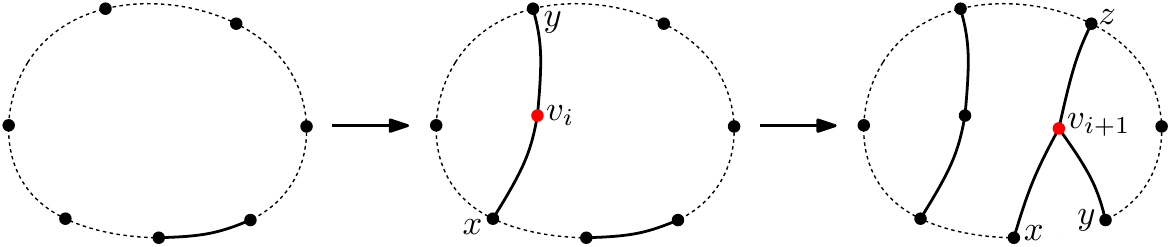}
  \caption{Two operations of a planar Henneberg construction: an ${\bf H_1}$-operation followed by an ${\bf H_2}$-operation.}
  \label{fig:Henneberg}
\end{figure}

\noindent Let $G$ be a planar Laman graph. From the fact that Laman graphs have $2|V|-3$ edges easily follows that $G$ contains a facial triangle. We choose an embedding of $G$ in which such a triangle $\{v_1,v_2,v_3\}$ is the outer face. We can assume that the outer face remains intact during a Henneberg construction, i.e., we never perform an ${\bf H_2}$-operation on an edge on the outer face. Let $v_1,v_2,v_3$ appear in this counterclockwise order around the outer triangle. We call $v_1,v_2$ the \emph{special vertices} and the outer edge $e^{*} = (v_1,v_2)$ the \emph{special edge} of $G$.

In the remainder of this section we describe three new combinatorial structures on $2$-connected plane graphs. Although we define the structures for general $2$-connected plane graphs, the most important structures (angular trees and edge labelings) exist only for plane Laman graphs.

%%%%%%%%%%%%%%%%%%%%%%%%%%%%%%%%%%%%%%%
%% A N G U L A R   S T R U C T U R E %%
%%%%%%%%%%%%%%%%%%%%%%%%%%%%%%%%%%%%%%%
\subsection{Angular Structure}\label{sec:angular-structure}

The \emph{angular graph $A_G$} of a plane graph $G$ is a plane bipartite graph defined as follows. The vertices of $A_G$ are the vertices $V(G)$ and faces $F(G)$ of $G$ and there exists an edge $(v,f)$ between $v \in V(G)$ and $f \in F(G)$ if and only if $v$ is incident to $f$. If $G$ is $2$-connected, then $A_G$ is a maximal bipartite planar graph and every face of $A_G$ is a quadrangle.
%The \emph{completion of $G$}, denoted by $\bar{G}$, is the union of $G$ and $A_G$, i.e., the maximally planar graph on the vertices and faces of $G$.
%
\begin{definition}
 An \emph{angular structure} of a $2$-connected plane graph $G$ with special edge $e^{*} = (v_1,v_2)$ is a set $T$ of edges of $A_G$ with the following two properties:
 \begin{description}
  \item[Vertex rule:] Every vertex $v \in V(G) \setminus \{v_1, v_2\}$ has exactly $2$ incident edges in $T$. Special vertices have no incident edge in $T$.
  \item[Face rule:] Every face $f \in F(G)$ has exactly $2$ incident edges not in $T$.
 \end{description}
\end{definition}
Let $S$ be the set of edges of $A_G$ that are not in $T$. The angular structure $T$ can be represented by orienting the edges of $A_G$ as follows. Every edge $(v,f)$ is oriented from $v$ to $f$ if $(v, f) \in T$, and from $f$ to $v$ if $(v, f) \in S$. This way every vertex of $A_G$ has exactly two outgoing edges (except for the special vertices). Such orientations of a maximal bipartite planar graph are called \emph{2-orientations} and have been introduced by de Fraysseix and Ossona de Mendez~\cite{deFraysseix200157}. It is also possible to derive an angular structure of $G$ from a 2-orientation of $A_G$. %In fact, the two notions are equivalent.
%the angular structures of $G$ and the 2-orientations of $A_G$ are in %bijection.
%
\begin{lemma}[\cite{deFraysseix200157}]\label{lem:angular-structure}
Every maximal bipartite planar graph has a 2-orientation. Thus every $2$-connected plane graph has an angular structure.
\end{lemma}
If $G$ is a Laman graph, then $|F(G)| = |V(G)| - 1$ by Euler's formula. Thus every angular structure $T$ consists of exactly $2|V(G)| - 4$ edges and spans exactly $|V(A_G)|-2 = 2|V(G)|-3$ vertices. Hence, if $T$ is connected, then $T$ is a spanning tree of $V(A_G) \setminus \{v_1,v_2\}$. An angular structure that is a tree is called an \emph{angular tree}. In Figure~\ref{fig:angular-structures} two angular structures of the same plane Laman graph are shown -- one being an angular tree.

%We will show later (c.f. Theorem~\ref{thm:Laman-iff-tree}) that a plane $2$-connected graph admits an angular tree if and only if it is a plane Laman graph. Moreover, we will show how to efficiently compute an angular tree of a plane Laman graph. Next, we derive some properties of a graph under the assumption that we have an angular tree of it at hand.

\begin{figure}[t]
 \centering
 \includegraphics{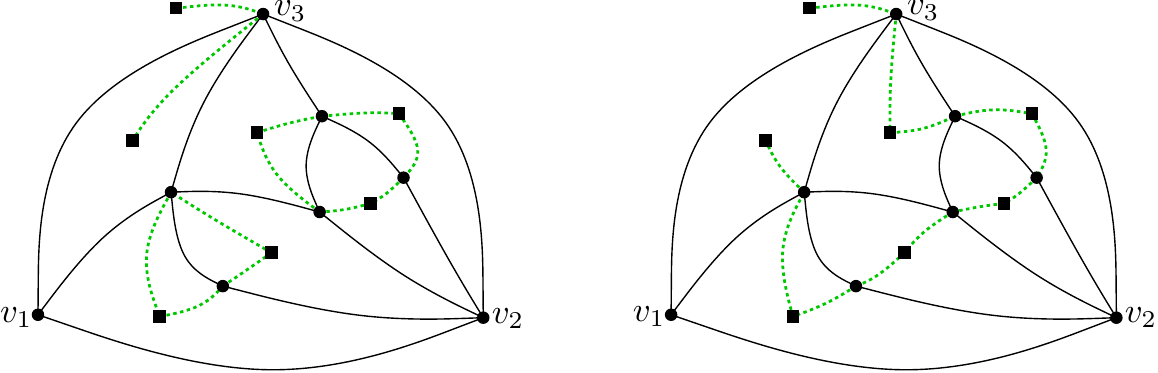}
 \caption{Two angular structures of the same plane Laman graph. The one on the right is an angular tree.}
 \label{fig:angular-structures}
\end{figure}

Next we show that every plane Laman graph admits an angular tree. Our proof is constructive and computes an angular tree along a planar Henneberg sequence of $G$. Consider a cycle $C$ in $A_G$ such that the edges of $C$ are alternatingly in $S$ and $T$. We say that $C$ is an \emph{alternating cycle}. We can perform a \emph{flip} on $C$ by removing all edges in $C \cap T$ from $T$ and adding all edges in $C \cap S$ to $T$. The resulting set of edges satisfies the properties of an angular structure. A flip corresponds to reversing the edges of a directed cycle in the corresponding $2$-orientation.

\begin{lemma}
\label{lem:fliptree}
Let $T$ consist of two connected components $A$ and $B$, where $A$ is a tree and $B$ contains a cycle. If we perform a flip on an alternating 4-cycle $C$ that contains an edge of $A$ and an edge of the cycle in $B$, then the resulting angular structure is a tree.
\end{lemma}
\begin{proof}
If we remove the edges in $C \cap T$ from $T$, then $B$ becomes a tree, and we split up $A$ into two trees $A_1$ and $A_2$. The edges in $C \cap S$ connect $A_1$ to $B$ and $A_2$ to $B$. The resulting angular structure is connected and hence a tree.
\end{proof}

\begin{theorem}\label{thm:tree-if-Laman}
Every plane Laman graph $G$ admits an angular tree $T$ and it can be computed in $\mathcal{O}(|V(G)|^2)$~time.
\end{theorem}
\begin{proof}
 We build $G$ and $T$ simultaneously along a planar Henneberg construction, which can be found in $\mathcal{O}(|V(G)|^2)$ time using an algorithm of Bereg~\cite{Bereg05}. $T$ remains a tree during the construction. We begin with the triangle $\{v_1,v_2,v_3\}$ and $T$ containing the two edges incident to $v_3$ in $A_G$. Now assume we insert a vertex $v$ into a face $f$ of $G$, which is split into two faces $f_1$ and $f_2$.

 For an ${\bf H_1}$-operation, let $x$ and $y$ be the original vertices of the graph. We add an edge $(u, f_1)$ to $T$ if and only if $u$ is incident to $f_1$ and $(u, f) \in T$ before the operation. We do the same for $f_2$. Furthermore, we add edges $(v, f_1)$ and $(v, f_2)$ to $T$. If $(x, f) \in T$ before the operation, then we remove either $(x, f_1)$ or $(x, f_2)$ from $T$. Similarly, if $(y, f) \in T$ before the operation, then we remove either $(y, f_1)$ or $(y, f_2)$ from $T$. By choosing these edges correctly, we can ensure that $f_1$ and $f_2$ satisfy the degree constraints; see Fig.~\ref{fig:AngTreeProof}(left). This operation cannot introduce a cycle, so $T$ must remain a tree.

\begin{figure}[t]
  \centering
  \includegraphics{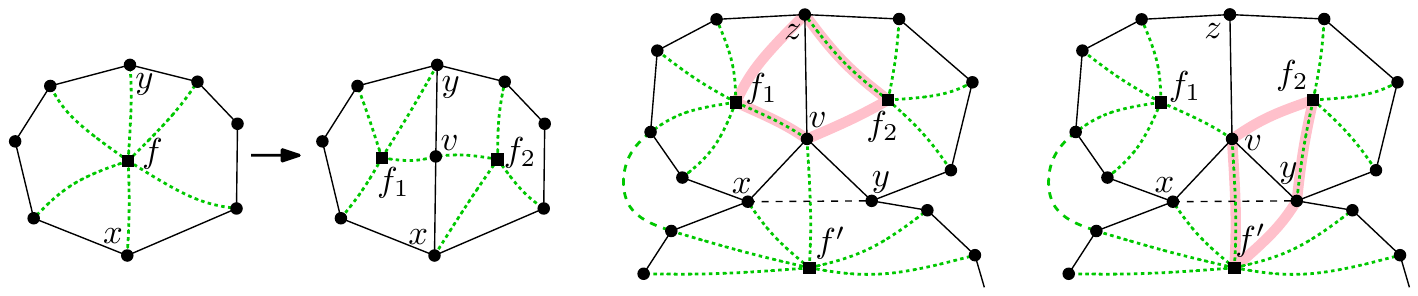}
  \caption{Left: updating $T$ (drawn dotted) for (${\bf H_1}$). Right: alternating cycles after (${\bf H_2}$).}
  \label{fig:AngTreeProof}
\end{figure}

 For an ${\bf H_2}$-operation, let $(x, y)$ and $z$ be the edge and vertex of the operation. Furthermore, let $f'$ be the face of $G$ that shares the edge $(x, y)$ with $f$ before the operation. We add an edge $(u, f_1)$ to $T$ if and only if $u$ is incident to $f_1$ and $(u, f) \in T$ before the operation (same for $f_2$). Furthermore, we add edges $(v, f')$ and either $(v, f_1)$ or $(v, f_2)$ to $T$. If $(z, f) \in T$ before the operation, then we remove either $(z, f_1)$ or $(z, f_2)$ from $T$. As above, we can choose the edges to ensure that $f_1$ and $f_2$ satisfy the degree constraints. However, this operation can introduce a cycle in $T$ containing the new vertex $v$ (if not, we are done). Assume w.l.o.g. that $f_1$ is part of this cycle, and hence $(v, f_1) \in T$; see Fig.~\ref{fig:AngTreeProof}(right).

If $(z, f_2) \in T$, then $(z, f_1) \notin T$, and the cycle formed by $(z, f_1)$, $(z, f_2)$, $(v, f_2)$, and $(v, f_1)$ is alternating and satisfies the requirements of Lemma~\ref{lem:fliptree}. We can flip this cycle to turn $T$ into a tree. If $(z, f_2) \notin T$, then $(y, f_2) \in T$ by the degree constraints on $f_2$. Also, $(y, f') \notin T$, for otherwise $T$ would contain a cycle before the operation. Thus, the cycle formed by $(y, f_2)$, $(y, f')$, $(v, f')$, and $(v, f_2)$ is alternating and satisfies the requirements of Lemma~\ref{lem:fliptree}. As before, we can flip this cycle to turn $T$ into a tree.

 At each step in the above procedure one vertex is added to $G$. The operations carried out to maintain the angular tree can be performed in $\mathcal{O}(1)$ time for an ${\bf H_1}$-operation and in $\mathcal{O}(|V(G)|)$ time for an ${\bf H_2}$-operation. Indeed, the bottleneck in the latter case is identifying the unique cycle in the intermediate angular structure. Thus the total runtime is $\mathcal{O}(|V(G)|^2)$, which concludes the proof.
\end{proof}

\noindent The following result is important for the construction of an L-contact representation of a plane Laman graph.

\begin{lemma}\label{lem:angular-matching}
If $T$ is an angular tree and $f$ is a triangular face of $G$, then $T$ contains a perfect matching between non-special vertices of $G$ and faces of $G$ different from $f$.
\end{lemma}
\begin{proof}
 Remove the vertex corresponding to $f$ (leaf in $T$) from $T$ and let $v$ be the non-special vertex with $(v,f) \in T$. Direct all edges of $T$ towards $v$. Now every face $f' \neq f$ has exactly one outgoing edge in $T$ and every non-special vertex has exactly one incoming edge in $T$. The desired matching can be obtained by matching each face different from $f$ to the unique endpoint $v \in V(G)$ of its outgoing edge in $T$.
\end{proof}

%%%%%%%%%%%%%%%%%%%%%%%%%%%%%%%%%%%%%%%
%%%%% A N G L E - L A B E L I N G %%%%%
%%%%%%%%%%%%%%%%%%%%%%%%%%%%%%%%%%%%%%%
\subsection{Angle Labeling}\label{sec:angle-labeling}

Next we define a labeling of the angles of $G$, using the angular structure above; see Fig.~\ref{fig:Labeling}. This labeling for $2$-connected plane graphs is similar to the Schnyder angle labeling for maximally plane graphs.

\begin{definition}
 An \emph{angle labeling} of a $2$-connected plane graph $G$ with special edge $e^{*} = (v_1,v_2)$ is a labeling of the angles of $G$ by $1,2,3,4$, with the following two properties:
 \begin{description}
  \item[Vertex rule:] Around every vertex $v \neq v_1,v_2$, in clockwise order, we get the following sequence of angles: exactly one angle labeled $3$, zero or more angles labeled $2$, exactly one angle labeled $4$, zero or more angles labeled $1$. All angles at $v_1$ are labeled $1$, all angles at $v_2$ are labeled $2$.
  \item[Face rule:] Around every face, in clockwise order, we get the following sequence of angles: exactly one angle labeled $1$, zero or more angles labeled $3$, exactly one angle labeled $2$, zero or more angles labeled $4$.
 \end{description}
\end{definition}

%\begin{figure}[htb]
%  \centering
%  \includegraphics{AngleLabeling}
%  \caption{The vertex rule and face rule of an angle labeling in relation to an edge labeling.}
%  \label{fig:AngleLabel}
%\end{figure}

\begin{theorem}\label{thm:angles-from-structure}
Every $2$-connected plane graph admits an angle labeling.
\end{theorem}
\begin{proof}
 By Lemma~\ref{lem:angular-structure} every $2$-connected plane graph $G$ admits an angular structure, which corresponds to a $2$-orientation of the angular graph $A_G$.
%Recall that $A_G$ is a maximal bipartite planar graph with white and black vertices %corresponding to vertices and faces of $G$, respectively.
%It is known~\cite{deFraysseix200157} that
The edges of a $2$-orientation can be colored in red and blue, such that the edges around each vertex $v$ are ordered as follows: one outgoing red edge, zero or more incoming red edges, one outgoing blue edge, zero or more incoming blue edges (the order is clockwise for $v \in V(G)$ and counterclockwise for $v \in F(G)$).
Such an orientation and coloring of the edges of a maximal bipartite planar graph is called a \emph{separating~decomposition}~\cite{deFraysseix200157}.

We now label each angle at a vertex $v$ of $G$ based on the color and orientation of the corresponding edge $(v,f)$ in the separating decomposition. If the edge is incoming at $v$ and colored blue, we label the angle $1$. If the edge is incoming at $v$ and colored red, we label the angle $2$. If the edge is outgoing at $v$ and colored red, we label the angle $3$. If the edge is outgoing at $v$ and colored blue, we label the angle $4$. It is now straightforward to verify that the vertex rule and face rule are implied by the order in which incident edges appear around each vertex in the separating decomposition.
\end{proof}

\noindent Note that the correspondence derived above between an angular structure $T$ and an angle labeling of $G$ is such that $(v, f) \in T$ if and only if the corresponding angle label is $3$ or $4$. Moreover, from an angle labeling one can derive the corresponding separating decomposition of $A_G$ and hence the corresponding angular structure. In particular, there is a bijection between angular structures of $G$ and angle labelings of $G$.

%%%%%%%%%%%%%%%%%%%%%%%%%%%%%%%%%%%%%%%
%%%%%% E D G E - L A B E L I N G %%%%%%
%%%%%%%%%%%%%%%%%%%%%%%%%%%%%%%%%%%%%%%
\subsection{Edge Labeling}\label{sec:edge-labeling}

Finally, we define an orientation and coloring of the edges of a $2$-connected plane graph $G$ based on an angular tree $T$ of $G$; see Fig.~\ref{fig:Labeling}.
This edge labeling for $2$-connected plane graphs is similar to the Schnyder edge labeling for maximally plane graphs.
%, which is similar to the well-known Schnyder woods of maximally plane graphs.

%Finally, we define an orientation and coloring of the edges of $G$, using the angle labeling (see Fig.~\ref{fig:Labeling}).
% which is similar to the well-known Schnyder woods of maximally plane graphs (see %Fig.~\ref{fig:Labeling}).
\begin{figure}[b]
  \centering
  \includegraphics{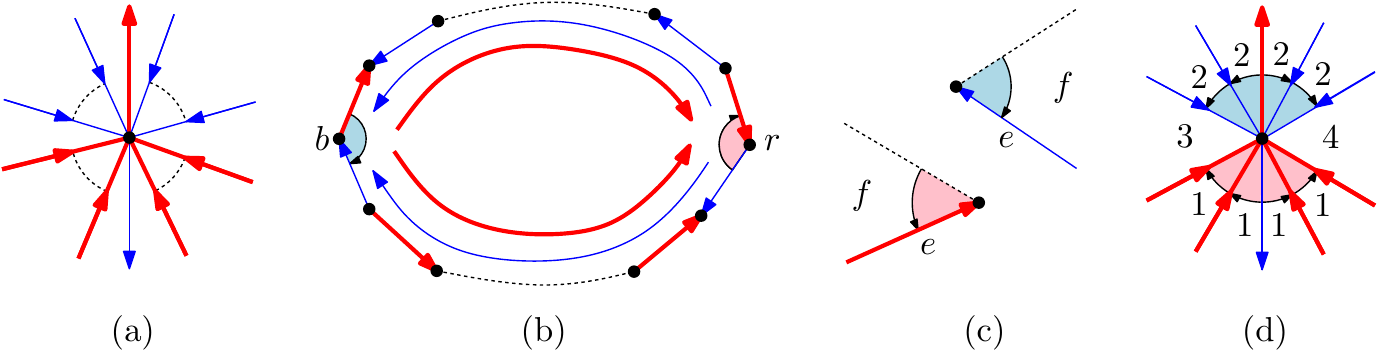}
  \caption{Vertex rule (a), face rule (b), and edge rule (c)-(d). Red edges are drawn thick.}
  \label{fig:Labeling}
\end{figure}
\begin{definition}
 An \emph{edge labeling} of a $2$-connected plane graph $G$ with special edge $e^{*} = (v_1,v_2)$ is an orientation and coloring of the non-special edges of $G$ with colors $1$ (red) and $2$ (blue), such that each of the following holds:
 \begin{description}
  \item[Vertex rule:] Around every vertex $v \neq v_1,v_2$, in clockwise order, we get the following sequence of edges: exactly one outgoing red edge, zero or more incoming blue edges, zero or more incoming red edges, exactly one outgoing blue edge, zero or more incoming red edges, and zero or more incoming blue edges. All non-special edges at $v_1$ are incoming and red, all non-special edges at $v_2$ are incoming and blue.
  \item[Face rule:] For every inner face $f$ there are two distinguished vertices $r$ and $b$. Every red edge on $f$ is directed from $b$ towards $r$, and every blue edge is directed from $r$ towards $b$. The vertices $r$ and $b$ are called the red and blue sink of $f$, respectively.
 \end{description}
 We denote the edge labeling by $(E_r,E_b)$, where $E_r$ and $E_b$ is the set of all red and blue edges, respectively.
\end{definition}
In an edge labeling $(E_r,E_b)$ of $G$ every non-special vertex has two outgoing edges. Together with the special edge this makes $2|V(G)|-3$ edges in total. Thus $|E(G)| = 2|V(G)|-3$ and $|F(G)| = |V(G)|-1$. Every inner face has exactly two sinks, which makes $2|F(G)| = |E(G)|-1$ in total. Indeed, there is a one-to-one correspondence between the non-special edges of $G$ and sinks of inner faces in $(E_r,E_b)$. We associate every directed edge $e$ with the inner face $f$ incident to it as illustrated in Fig.~\ref{fig:Labeling}(d). This way we have the following for every edge labeling $(E_r,E_b)$ of $G$.
%
%\begin{figure}[htb]
% \centering
% \includegraphics{edge-rule.pdf}
% \caption{the Edge Rule}
% \label{fig:edge-rule}
%\end{figure}

\begin{description}
 \item[Edge rule:] Every non-special edge $e$ corresponds to one incident inner face $f$, such that the endpoint of $e$ is a sink of $f$ in the color of $e$.
\end{description}

\begin{theorem}\label{thm:edge-from-angle}
 If a $2$-connected plane graph admits an angular tree, then it admits an edge labeling.
\end{theorem}
\begin{proof}
 Let $G$ be a $2$-connected plane graph and $T$ be an angular structure of $G$. By Theorem~\ref{thm:angles-from-structure}, $G$ admits an angle labeling that corresponds to $T$, i.e., the angle of a face $f$ at a vertex $v$ is labeled $3$ or $4$ if and only if $(v,f) \in T$. We split every vertex $v$ in $G$, except for $v_1$ and $v_2$, into two vertices $v^1$ and $v^2$, in such a way that for $i=1,2$ all edges incident to an angle labeled $i$ are incident to $v^i$. We call the resulting graph $H$. In other words $H$ arises from $G$ by splitting each non-special vertex along its two edges in $T$. Thus, as $T$ is acyclic, $H$ is connected. Since $H$ consists of $2|V(G)|-2$ vertices ($v_1,v_2$ plus $2(|V(G)|-2)$ split vertices) and $|E(G)| = 2|V(G)|-3$ edges, $H$ is a tree.

 We orient every edge $e$ in $H$ towards the special edge $e^*$ of $G$. We color $e$ red if it is outgoing at some $v^2$ and blue if it is outgoing at some $v^1$. It is now straightforward to check, using the vertex rule and face rule of the angle labeling, that this orientation and coloring of all non-special edges is indeed a valid edge labeling of $G$.
\end{proof}

\noindent
 Not every edge labeling corresponds to an angular tree. Furthermore, some but not all angular structures that are not trees correspond to an edge labeling. For example, the angular structure in Figure~\ref{fig:angular-structures} (left) does not have a corresponding edge labeling. Hence, edge labelings of $G$ and angular structures (or angular trees) of $G$ are \emph{not} in bijection.
 %Moreover, from every edge labeling one can easily derive an angular structure and angle labeling. Hence all three concepts are in bijection provided the angular structure is a tree.
%We suspect that this is also the case for general angular structures. We just do not know %how to derive an edge labeling from a general angle labeling.

\begin{theorem}\label{thm:acyclic-and-trees}
 An edge labeling $(E_r,E_b)$ of a $2$-connected plane graph $G$ with special edge $e^{*} = (v_1,v_2)$ has the following two properties:
 \begin{enumerate}[label=(\roman*)]
  \item The graph $E_r \cup E_b^{-1}$ ($E_b \cup E_r^{-1}$) is acyclic, where $E_b^{-1}$ is $E_b$ with the direction of all edges reversed.\label{enum:acyclic}
  \item The graph $E_r$ ($E_b$) is a spanning tree of $G \setminus \{v_2\}$ ($G \setminus \{v_1\}$) with all edges directed towards $v_1$ ($v_2$).\label{enum:tree}
 \end{enumerate}
\end{theorem}
\begin{proof}
 Consider the graph $E_r \cup E_b^{-1}$. Since every vertex except for $v_1$ and $v_2$ has an outgoing red edge and an outgoing blue edge, there is only one source (all edges are outgoing at $v_2$) and one sink (all edges are incoming at $v_1$) in $E_r \cup E_b^{-1}$. By the face rule, every face has exactly one source (the blue sink) and one sink (the red sink). The face rule for the inner face of $G$ containing the special edge $e^{*}$ implies that the outer cycle as well has exactly one source ($v_1$) and exactly one sink ($v_2$). Every nesting minimal (the set of faces it circumscribes is inclusion minimal) directed cycle in a plane graph is either a facial cycle or has a source or sink in its interior. This proves~\ref{enum:acyclic}. Part~\ref{enum:tree} follows directly from part~\ref{enum:acyclic} and the fact that every non-special vertex has one outgoing edge in $E_r$ ($E_b$).
\end{proof}

\section{L-Contact Graphs}\label{sec:L-contact}

\begin{figure}[b]
  \centering
  \vspace{-.5\baselineskip}
  \includegraphics{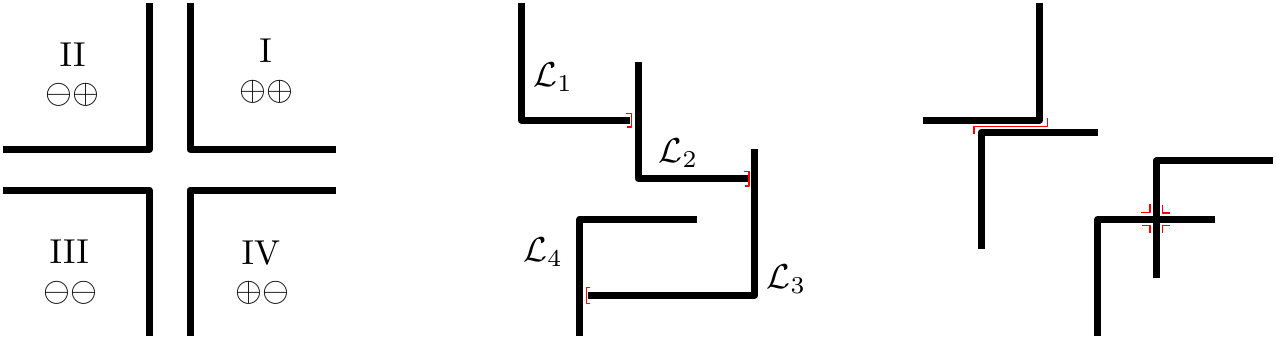}
  \caption{Left: possible L-shapes. Middle: valid contacts. Right: invalid contacts.}
  \label{fig:Lcontact}
\end{figure}

%In this section we consider \emph{L-contact graphs} and \emph{proper L-contact graphs}. The main result of this section is the following.
%
%\begin{theorem}\label{thm:Laman-are-L-contact}
% Plane Laman graphs are precisely proper L-contact graphs.
%\end{theorem}
%
An \emph{L-shape} $\LL$ is a path consisting of exactly one horizontal segment and exactly one vertical segment. There are four different types of L-shapes; see Fig.~\ref{fig:Lcontact}(left). Two L-shapes $\LL_1$ and $\LL_2$ make \emph{contact} if and only if the endpoint of one of the two L-shapes coincides with an interior point of the other L-shape; see Fig.~\ref{fig:Lcontact}(middle). If the endpoint belongs to $\LL_1$, then we say that $\LL_1$ makes contact with $\LL_2$. Note that we do not allow contact using the bend of an L-shape; see Fig.~\ref{fig:Lcontact}(right).

A graph $G = (V, E)$ is an \emph{L-contact graph} if there exist non-crossing L-shapes $\LL(v)$ for each $v \in V$, such that $\LL(u)$ and $\LL(v)$ make contact if and only if $(u, v) \in E$. We call these L-shapes the \emph{L-contact representation} of $G$. We can match edges of L-contact graphs to endpoints of L-shapes. However, an endpoint that is bottommost, topmost, leftmost, or rightmost cannot correspond to an edge. We call an L-contact representation \emph{maximal} if every endpoint that is neither bottommost, topmost, leftmost, nor rightmost makes a contact, and there are at most three endpoints that do not make a contact. We assume that the bottommost, topmost, leftmost, and rightmost endpoints are uniquely defined.

%Note that two L-shapes may have two contacts, which we interpret as two parallel edges between the two corresponding vertices in $G$.

%As in the case of 2-segment representations there is an injection $\phi$ from edges in the graph to endpoints of L-shapes. However, one bottommost and one topmost, as well as one leftmost and one rightmost, endpoint is not in the image of $\phi$. Note that this could be only two endpoints in total. Let us call an L-contact representation of a connected graph $G$ \emph{maximal} if every endpoint that is neither bottommost, topmost, leftmost, nor rightmost is in the image of $\phi$, i.e., corresponds to a contact of two L-shapes.
%
%\medskip
%
%\textbf{Obvious lemma, but nice to state anyway.}
%
%\medskip
%
%\begin{lemma}\label{lem:subgraph-closed}
%If a graph $G$ is an L-contact graph, then every subgraph $G'$ of $G$ is also an L-contact graph.
%\end{lemma}
%\begin{proof}
%Given an L-contact representation of $G$, we can construct an L-contact representation of $G'$ as follows. First remove all L-shapes $\LL(v)$ for vertices not in $G'$. For every remaining edge $(u, v)$ not in $G'$, assume that $\LL(u)$ makes contact with $\LL(v)$. We simply shorten the respective segment of $\LL(u)$ such that the contact is removed. The resulting L-shapes are an L-contact representation of $G'$.
%\end{proof}

In a maximal L-contact representation of a graph $G$, each inner face of $G$ is bounded by a simple rectilinear polygon, which is contained in the union of all L-shapes. Now each $\LL(v)$ has a right angle, which is a convex corner of the polygon corresponding to one incident face at $v$ and a concave corner corresponding to another incident face at $v$, provided the corresponding face is an inner face.
%
%\begin{figure}[htb]
%  \centering
%  \includegraphics{K4}
%  \caption{Left: $K_4$ is an L-contact graph but not a Laman graph. Right: Illustration of the proof of Lemma~\ref{lem:L-and-Seg}.}
%  \label{fig:K4}
%\end{figure}
%
\begin{lemma}\label{lem:L-and-Seg}
If a graph $G$ has a maximal L-contact representation in which each inner face contains the right angle of exactly one $\LL$, then $G$ is a plane Laman~graph.
\end{lemma}
\begin{proof}
 Consider a maximal L-contact representation of $G$ in which every inner face contains the right angle of exactly one $\LL$. By the definition of maximal L-contact representations, we get that $|E(G)| \geq 2|V(G)| - 3$. We need to show that $|E(W)| \leq 2|W| - 3$ for all subsets $W \subseteq V(G)$ of at least two vertices. For the sake of contradiction, let $W$ be a set ($|W| \geq 2$) with $|E(W)| \geq 2|W| - 2$. It follows that at most two endpoints of L-shapes corresponding to vertices in $W$ do not make contact when restricted to $W$. Since this holds for one bottommost endpoint and one topmost endpoint, we have $|E(W)| = 2|W| - 2$. Moreover, if we choose $W$ to be inclusion-minimal among all such sets, then $G(W)$ is $2$-connected; see thick L-shapes corresponding to $W$ in Fig.~\ref{fig:K4}.

 The outer face of $G(W)$ is bounded by a rectilinear polygon $\mathcal{P}$ with two additional ends sticking out. This polygon is highlighted in Fig.~\ref{fig:K4}. Consider the vertex set $W' \supseteq W$ of all vertices whose corresponding L-shapes are contained in $\mathcal{P}$, i.e., $G(W)$ is a subgraph of $G(W')$ and every inner face of $G(W')$ is an inner face of $G$. Since the representation is maximal we have $|E(W')| = |E(W)| + 2|W'\setminus W| = 2|W'| - 2$. Hence $G(W')$ has too many edges as well. We want to show that one inner face of $G(W')$ has two convex angles, which would then complete the proof.

 Let $k$ be the number of outer vertices of $G(W')$. Since $\mathcal{P}$ has only two endpoints sticking out, all but two of its convex corners are due to a single $\LL$, i.e.,
 \begin{displaymath}
  \# \textnormal{convex corners of } \mathcal{P} \leq k + 2.
 \end{displaymath}
 Each outer edge of $G(W')$, except for two, corresponds to a contact that is a concave corner of $\mathcal{P}$, i.e.,
 \begin{displaymath}
  \# \textnormal{concave corners of } \mathcal{P} \geq k - 2.
 \end{displaymath}
 In every rectilinear polygon the number of concave corners is exactly the number of its convex corners minus four. Thus we conclude that both inequalities above must hold with equality. In particular, every concave corner of $\mathcal{P}$ corresponds to a contact of two L-shapes and no concave corner is due to a single $\LL$. Moreover, every L-shape corresponding to an outer vertex in $G(W')$ forms a convex corner of $\mathcal{P}$. Hence for every $w \in W'$ the right angle of $\LL(w)$ lies inside $\mathcal{P}$.

 By Euler's formula $G(W')$ has precisely $|W'| - 1$ inner faces. Since there are $|W'|$ right angles among those inner faces, one inner face must have two right angles.
\end{proof}

\begin{figure}[t]
  \centering
  \includegraphics{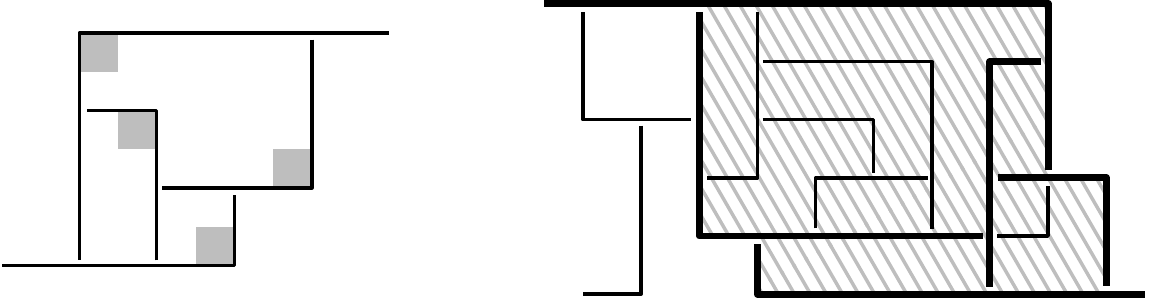}
  \caption{Left: $K_4$ is an L-contact graph but not a Laman graph. Right: Illustration of the proof of Lemma~\ref{lem:L-and-Seg}.}
  \label{fig:K4}
\end{figure}

\begin{definition}
A maximal L-contact representation is \emph{proper} if every inner face contains the right angle of exactly one $\LL$. An L-contact graph is \emph{proper} if it has a proper L-contact representation.
\end{definition}

%\begin{definition}
%A maximal L-contact representation is \emph{proper} if every inner face contains the right angle of exactly one $\LL$. An L-contact graph \emph{proper} if it has a proper L-contact representation, or is a subgraph of one with such a proper representation.
%\end{definition}

\noindent Lemma~\ref{lem:L-and-Seg} states that all proper L-contact graphs are plane Laman graphs. The main result of the remainder of this section is the following.
\begin{theorem}\label{thm:Laman-are-L-contact}
Plane Laman graphs are precisely proper L-contact graphs.
\end{theorem}
To obtain an L-contact representation of a plane Laman graph, we require only the existence of an angular tree with the corresponding edge labeling. Thus, if a $2$-connected plane graph $G$ admits an angular tree, then it has a corresponding edge labeling by Theorem~\ref{thm:edge-from-angle}, and we can compute a proper L-contact representation of $G$. We obtain the following characterization of planar Laman graphs as a by-product of our approach.

\begin{theorem}\label{thm:Laman-AngTree}
A planar $2$-connected graph is a Laman graph if and only if it admits an angular tree.
\end{theorem}

%Then Lemma~\ref{lem:L-and-Seg} states that all proper L-contact graphs are 2-segment graphs. The main result of the paper is that the converse is true as well, i.e., every 2-segment graph is a proper L-contact graph. Thus, those two graph classes coincide.
%
%\begin{lemma}
% Every proper L-contact representation of a plane Laman graph $G$ inherits an edge labeling of $G$.
%\end{lemma}
%\begin{proof}
% Consider two L-shapes $\LL(u)$ and $\LL(v)$ that make contact. If $\LL(u)$ makes contact with $\LL(v)$, then we direct the edge $(u, v)$ towards $v$. If this contact is made with the vertical edge of $\LL(u)$, then we color the edge red, otherwise we color the edge blue.
%
% TU: \textbf{May be nice to prove this ... but may as well be not ?!}
%\end{proof}

%%%%%%%%%%%%%%%%%%%%%%%%%%%%%%%%%%%%%%%
%%%%%%% V E R T E X   T Y P E S %%%%%%%
%%%%%%%%%%%%%%%%%%%%%%%%%%%%%%%%%%%%%%%
\subsection{Vertex Types}\label{sec:types}

Assume we have an angular tree $T$ with corresponding edge labeling $(E_r, E_b)$ for a plane Laman graph $G$. Every non-special vertex $v$ in $G$ has two incident edges in $T$. The other endpoint of such an edge corresponds to a face in $G$. These are the two faces that contain the bend of $\LL(v)$. The matching $M$ of $T$ obtained from Lemma~\ref{lem:angular-matching} (using the outer face of $G$ as the triangular face) determines for every vertex of $G$ the incident inner face $f$ containing the right angle of $\LL(v)$. The outgoing red (blue) edge of a vertex $v$ determines the contact made by the horizontal (vertical) leg of $\LL(v)$.

We derive from $M$ and $(E_r, E_b)$ the type of the L-shape $\LL(v)$ for every vertex $v$. The \emph{red sign} and \emph{blue sign} of a vertex $v$, denoted by $t_r(v)$ and $t_b(v)$, represent the direction of the horizontal and vertical leg of $\LL(v)$, respectively. We write the type of $v$ as $t(v) = t_r(v)t_b(v)$, or as its quadrant number (see Fig.~\ref{fig:Lcontact} left).

%We reformulate the types according to the quadrant they represent. The \emph{red sign} and \emph{blue sign} of a vertex $v$, denoted by $t_r(v)$ and $t_b(v)$, stand for the sign of the $x$-coordinate and $y$-coordinate of points in the corresponding quadrant, respectively. If $t(v) = t_r(v)t_b(v)$ then we have the following correspondence:
%%
%\begin{align*}
% \text{vertex } v \text{ has type I } \quad &\longleftrightarrow \quad t(v) = \I\\
% \text{vertex } v \text{ has type II } \quad &\longleftrightarrow \quad t(v) = \II\\
% \text{vertex } v \text{ has type III } \quad &\longleftrightarrow \quad t(v) = \III\\
% \text{vertex } v \text{ has type IV } \quad &\longleftrightarrow \quad t(v) = \IV
%\end{align*}
%
First we set $t_b(v_1) = \oplus$ and $t_r(v_2) = \oplus$ (the red sign of $v_1$ and the blue sign of $v_2$ are irrelevant). For every non-special vertex $v$, let $e_r(v)$ ($e_b(v)$) be its outgoing red (blue) edge, and $e_M(v)$ its incident edge in $M$. The angle between $e_r(v)$ and $e_b(v)$ that contains $e_M(v)$ is called the \emph{matched angle}. The opposite angle is called the \emph{unmatched angle} ($v_1$ and $v_2$ have only an unmatched angle). We set the types according to the following rule.
\smallskip
\begin{description}
 \item[Type rule:] Let $e = (u,v)$ be a directed edge from $u$ to $v$ of color $c$. If $e$ lies in the unmatched angle of $v$, we set $t_c(u) = t_c(v)$, otherwise $t_c(u) \neq t_c(v)$.
\end{description}
\smallskip
%
%We set $t_c(u) = t_c(v)$ if $e$ lies in the unmatched angle of $v$, and $t_c(u) \neq t_c(v)$ otherwise.
%We call $v$ \emph{odd} if $e_r(v),e_M(v),e_b(v)$ appear in clockwise order around $v$ in the completion $\bar{G}$, and $v$ \emph{even} otherwise.
We need to check if this type rule, along with $T$, $M$, and $(E_r, E_b)$, results in a correct L-contact representation. Around every vertex $v$, the neighboring vertices with incoming edges to $v$ must have the correct red or blue sign. For example, if $t(v) = \text{I}$ and the edge $u \to v$ is blue and lies in the matched angle of $v$, then $t_b(u) = \ominus$. Note that this follows directly from the type rule (see Fig.~\ref{fig:VertexTypes} left).

Secondly, the convex angle of an L-shape $\LL(v)$ must belong to the face that contains $e_M(v)$. For example, if $e_b(v),e_M(v),e_r(v)$ appear in clockwise order around $v$, then $t(v) = \text{I}$ or $t(v) = \text{III}$. We say $v$ is \emph{odd} if $e_b(v),e_M(v),e_r(v)$ appear in clockwise order around $v$, and \emph{even} otherwise.
\begin{figure}[b]
  \centering
  \vspace{-.5\baselineskip}
  \includegraphics{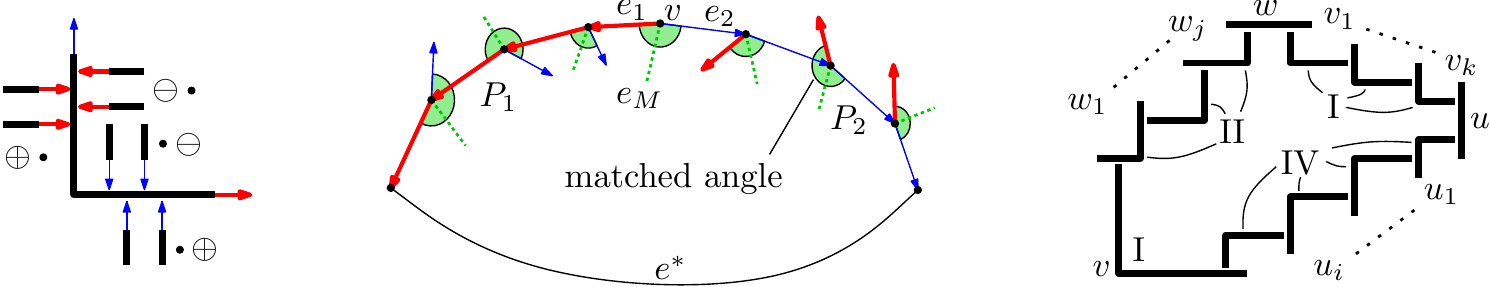}
  \caption{Left/Right: types around a vertex/face ($t(v) = \text{I}$). Middle: proof Lemma~\ref{lem:parity-type}.}
  \label{fig:VertexTypes}
\end{figure}
\begin{lemma}\label{lem:parity-type}
 A non-special vertex $v$ is odd if and only if $t_r(v) = t_b(v)$.
\end{lemma}
\begin{proof}
 Consider the directed red path $P_1$ from $v$ to $v_1$ and the directed blue path $P_2$ from $v$ to $v_2$ (see Fig.~\ref{fig:VertexTypes} middle). Since $E_r \cup E_b^{-1}$ is acyclic by Theorem~\ref{thm:acyclic-and-trees}, $P_1 \cap P_2$ consists only of $v$. Let $C$ be the cycle formed by $P_1,P_2$ and the special edge $e^{*}$, and $G'$ be the maximal subgraph of $G$ whose outer cycle is $C$. We define $r_1,r_2,r_3,r_4$ as follows (we define $b_1,b_2,b_3,b_4$ analogously w.r.t. $P_2$):
 \begin{itemize}[label=]
  \item $r_1 := \#\{ e = (u,v) \in P_1 \,|\, e \text{ in unmatched angle of } v \text{ and } e_M(v) \text{ outside } G'\}$
  \item $r_2 := \#\{ e = (u,v) \in P_1 \,|\, e \text{ in unmatched angle of } v \text{ and } e_M(v) \text{ inside } G'\}$
  \item $r_3 := \#\{ e = (u,v) \in P_1 \,|\, e \text{ in matched angle of } v \text{ and } e_M(v) \text{ outside } G'\}$
  \item $r_4 := \#\{ e = (u,v) \in P_1 \,|\, e \text{ in matched angle of } v \text{ and } e_M(v) \text{ inside } G'\}$
 \end{itemize}
%
 %We define $b_1,b_2,b_3,b_4$ analogously w.r.t. $P_2$.
 Now let $k = |C|$ be the number of vertices on $C$ and $|V(G')| = k + n'$. Then $G'$ has $2n' + k + r_2 + r_3 + b_2 + b_3$ edges and thus by Euler's formula $n' + b_2 + b_3 + r_2 + r_3 + 1$ inner faces. On the other hand $G'\setminus \{v\}$ contains exactly $n' + b_2 + b_4 + r_2 + r_4$ matching edges. So if $v$ is odd,
 %i.e., $e_r,e_M,e_b$ appear in this clockwise order around $v$,
 then $e_M$ lies inside $G'$, too. Since the number of inner faces and matching edges must coincide we have $b_3 + r_3 = b_4 + r_4$. In particular $b_3 + b_4$ and $r_3 + r_4$ have the same parity, which means that the red and blue sign of $v$ coincide. If $v$ is even,
 %i.e., $e_r,e_M,e_b$ appear in this \emph{counter}clockwise order around $v$,
 then $e_M$ lies outside $G'$ and we get $b_3 + r_3 + 1 = b_4 + r_4$, which implies that $b_3 + b_4$ and $r_3 + r_4$ have different parity. Hence the red sign and blue sign at $v$ are distinct.
\end{proof}
Finally we consider the faces in the L-contact representation. Every inner face $f$ of $G$ has three special vertices: the two sinks $u$ and $w$, as well as the vertex $v$ that $f$ is matched to in $M$. Let $u,u_1,\ldots,u_i,v,w_1,\ldots,w_j,w$,\linebreak $v_1,\ldots,v_k$ be the clockwise order of the vertices around $f$. The type rule implies the following shape of faces in the L-contact representation (see Fig.~\ref{fig:VertexTypes} right).
\begin{lemma}\label{lem:types-around-face}
 Let $v$ be the vertex that is matched to a face $f$, and $t$ be the type of $v$. Then we have the following:
 \begin{itemize}
  \item Each of $u_1,\ldots,u_i$ has type $t - 1$.
  \item Each of $v_1,\ldots,v_k$ has type $t$.
  \item Each of $w_1,\ldots,w_j$ has type $t + 1$.
 \end{itemize}
\end{lemma}
\begin{proof}
 Let us assume that $v$ is odd, i.e., $e_r,e_M,e_b$ appear around $v$ in this counterclockwise order. (The case that $v$ is even is analogous.) Then the face rule and vertex rule imply that $v$, the blue sink of $f$, and the red sink of $f$ appear around $f$ in this clockwise order, i.e., $u$ and $w$ are the red and blue sink of $f$, respectively. Every vertex incident to $f$, except for $u,v,w$, has an edge with $f$ in the angular tree $T$ but not in the matching $M$. This means that $f$ lies in the unmatched angle of each such vertex. Together with the face rule of the edge labeling, this implies that $u_1,\ldots,u_i$ and $w_1,\ldots,w_j$ are even, while $v_1,\ldots,v_k$ are odd.

 Consider the edge $e$ between $v$ and $w_1$. If $e$ is blue, then it is directed from $v$ to $w_1$ and lies in the unmatched angle of $w_1$. Hence $t_b(w_1)=t_b(v)$. Since $w_1$ is even, its type is indeed $t+1$. If $e$ is red, then it is directed from $w_1$ to $v$ and lies in the matched angle of $v$. Hence $t_r(w_1) \neq t_r(v)$. Again, since $w_1$ is even, its type is $t+1$. Every edge between $w_l$ and $w_{l+1}$ ($l < i$) lies in the unmatched angle of its endpoint and hence both vertices have the same red or blue sign. Since both are even, they have in fact the same type. Similarly, one can show that each of $u_1,\ldots,u_i$ has type $t - 1$.

 Now consider $w = v_0,v_1,\ldots,v_k,v_{k+1}=u$. Again all edges among those vertices that end at $v_l$ ($1 \leq l \leq k$) lie in the unmatched angle of their endpoint. Hence since $v_1,\ldots,v_k$ are all odd, they have the same type (either $t$ or $t+2$). For $w$, the blue sink of $f$, we have three cases. If both edges at $w$ are incoming (and hence blue), then both lie in the same angle (either matched or unmatched) of $w$ and thus $t_b(v_1)=t_b(w_j)$, which implies $t(v_1)=t(v)$. If $w \to v_1$ is red and $w_j \to w$ is incoming blue and lies in the matched angle of $w$, then $w$ is even. It follows that $t_b(w_j) \neq t_b(w)$ and $t_r(w_j) \neq t_r(w) = t_r(v_1)$. This again implies $t(v_1)=t(v)$. If $w_j \to w$ lies in the unmatched angle of $w$, then $w$ is odd and we have $t_b(w_j) = t_b(w) = t_r(w) = t_r(v_1) = t_b(v_1)$ as desired. Similarly, if $w \to w_j$ is red and $v_1 \to w$ lies in the matched angle of $w$, then $w$ is odd and $t_r(w) = t_r(w_j)$, $t_b(w) \neq t_b(w_j)$, and $t_b(w) \neq t_b(v_1)$. Thus $t_b(w_j) = t_b(v_1)$, which implies $t(v_1) = t(v)$. Finally if $v_1 \to w$ lies in the unmatched angle of $w$, then $w$ is even and we have $t(w_j) = t(w)$, and $t_b(w_j) = t_b(w) = t_b(v_1)$ as desired.
\end{proof}
%
% The following result follows immediately from the type rule; see Fig.~\ref{fig:VertexTypes}.
%
% \begin{lemma}\label{lem:types-around-vertex}
%  If $v$ is even and has type $t$, then the types of the incoming neighbors of $v$ are as follows:
%  \begin{itemize}
%   \item Endpoints of incoming blue edges between $e_r(v)$ and $e_M(v)$ have type $t-2$ or $t-1$.
%   \item Endpoints of the remaining incoming blue edges have type $t$ or $t+1$.
%   \item Endpoints of incoming red edges between $e_b(v)$ and $e_M(v)$ have type $t+1$ or $t+2$.
%   \item Endpoints of the remaining incoming red edges have type $t-1$ or $t$.
%  \end{itemize}
%  If $v$ is odd, then the same is true when exchanging the roles of red and blue.
% \end{lemma}
%\begin{figure}[htb]
% \centering
% %\includegraphics{types-around-vertex.pdf}
% \caption{Illustration of the statement of Lemma~\ref{lem:types-around-vertex}.}
% \label{fig:types-around-vertex}
%\end{figure}
%

%%%%%%%%%%%%%%%%%%%%%%%%%%%%%%%%%%%%%%%
%%%%%%% I N E Q U A L I T I E S %%%%%%%
%%%%%%%%%%%%%%%%%%%%%%%%%%%%%%%%%%%%%%%
\subsection{Inequalities}\label{sec:inequalities}

Given the type of every vertex $v$, it suffices to find the point $(x(v),y(v)) \in \mathbb{R}^2$ where the bend of $\LL(v)$ is located. Additionally we define for each inner face $f$ an auxiliary point $(x(f),y(f)) \in \mathbb{R}^2$, which in the L-contact representation of $G$ will correspond to some point in the bounded region corresponding to $f$.

We use two directed (multi-)graphs $D_r$ and $D_b$ on the vertices and inner faces of $G$ to describe inequalities for the $x$- and $y$-coordinates, respectively. For every inequality $x(u) < x(v)$ ($y(u) < y(v)$) there is an edge $u \to v$ in $D_r$ ($D_b$), where $u,v \in V(G)\cup F(G)$. Both graphs $D_r$ and $D_b$ contain all edges of $G$. The direction of an edge $(u, v)$ can be determined by $t(u)$, $t(v)$, and $(E_r, E_b)$. An edge $u \to v$ is in $D_r$ iff (i) $u \to v \in E_r$ and $t_r(u) = \oplus$, (ii) $v \to u \in E_r$ and $t_r(v) = \ominus$, (iii) $u \to v \in E_b$ and $t_r(v) = \ominus$, or (iv) $v \to u \in E_b$ and $t_r(u) = \oplus$. Similarly, $u \to v$ is in $D_b$ iff (i) $u \to v \in E_b$ and $t_b(u) = \oplus$, (ii) $v \to u \in E_b$ and $t_b(v) = \ominus$, (iii) $u \to v \in E_r$ and $t_b(v) = \ominus$, or (iv) $v \to u \in E_r$ and $t_b(u) = \oplus$.

The special edge $e^* = (v_1,v_2)$ is directed $v_2 \to v_1$ in $D_r$ and $v_1 \to v_2$ in $D_b$. Note that this is consistent with the above rules using $t_b(v_1) = \oplus = t_r(v_2)$ and putting either $v_1 \to v_2$ into $E_b$ or $v_2 \to v_1$ into $E_r$.

We can derive from the face rule of the edge labeling $(E_r,E_b)$, the types around a face (Lemma~\ref{lem:types-around-face}), and the definition of directed edges above, how the edges of a face $f$ of the Laman graph $G$ are oriented in $D_r$ and $D_b$. We remark that some edges of $D_r, D_b$, namely those between faces and vertices of $G$, are yet to be defined, and that a facial cycle in $G$ will correspond to a non-facial cycle in $D_r$, as well as $D_b$.

Let $f$ be an inner face in $G$ with its three distinguished vertices $u,v,w$. For convenience we put $u_{i+1} = v = w_0$, $w_{j+1} = w = v_0$, and $v_{k+1} = u = u_0$. Then all but two edges of $f$ appear in $D_r$ and $D_b$ according to the following table:

\begin{table}[h]
 \centering
 \newcommand{\la}{\ensuremath{\leftarrow}}
 \begin{tabular}{|c|c|c|c|c|}
  \hline
  & $t(v) = \I$ & $t(v) = \II$ & $t(v) = \III$ & $t(v) = \IV$ \\
  \hline
  \multirow{3}{*}{$D_r$} & $u_0 \la \ldots \la u_{i+1}$ & $u_1 \to \ldots \to u_{i+1}$ & $u_0 \to \ldots \to u_{i+1}$ & $u_1 \la \ldots \la u_{i+1}$ \\
  & $v_1 \to \ldots \to v_{k+1}$ & $v_0 \to \ldots \to v_k$ & $v_1 \la \ldots \la v_{k+1}$ & $v_0 \la \ldots \la v_k$ \\
  & $w_0 \to \ldots \to w_j$ & $w_0 \la \ldots \la w_{j+1}$ & $w_0 \la \ldots \la w_j$ & $w_0 \to \ldots \to w_{j+1}$ \\
  \hline
  \multirow{3}{*}{$D_b$} & $u_1 \la \ldots \la u_{i+1}$ & $u_0 \la \ldots \la u_{i+1}$ & $u_1 \to \ldots \to u_{i+1}$ & $u_0 \to \ldots \to u_{i+1}$ \\
  & $v_0 \la \ldots \la v_k$ & $v_1 \to \ldots \to v_{k+1}$ & $v_0 \to \ldots \to v_k$ & $v_1 \la \ldots \la v_{k+1}$ \\
  & $w_0 \to \ldots \to w_{j+1}$ & $w_0 \to \ldots \to w_j$ & $w_0 \la \ldots \la w_{j+1}$ & $w_0 \la \ldots \la w_j$ \\
  \hline
 \end{tabular}

 \medskip

 \caption{The edges of a face $f$ of $G$ (except for two) form in $D_r$ and $D_b$ three directed paths.}
 \vspace{-2\baselineskip}
 \label{tab:inequalities-around-face}
\end{table}

\noindent
The situation around a face $f$ is illustrated in Figure~\ref{fig:ineq-around-face2}. The two edges $e,e'$ in $f$ that are not listed for $D_r$ ($D_b$) in Table~\ref{tab:inequalities-around-face} are incident to the blue (red) sink of $f$. These are the dashed edges in the figure.

\begin{figure}[htb]
 \centering
 \includegraphics[width=.9\textwidth]{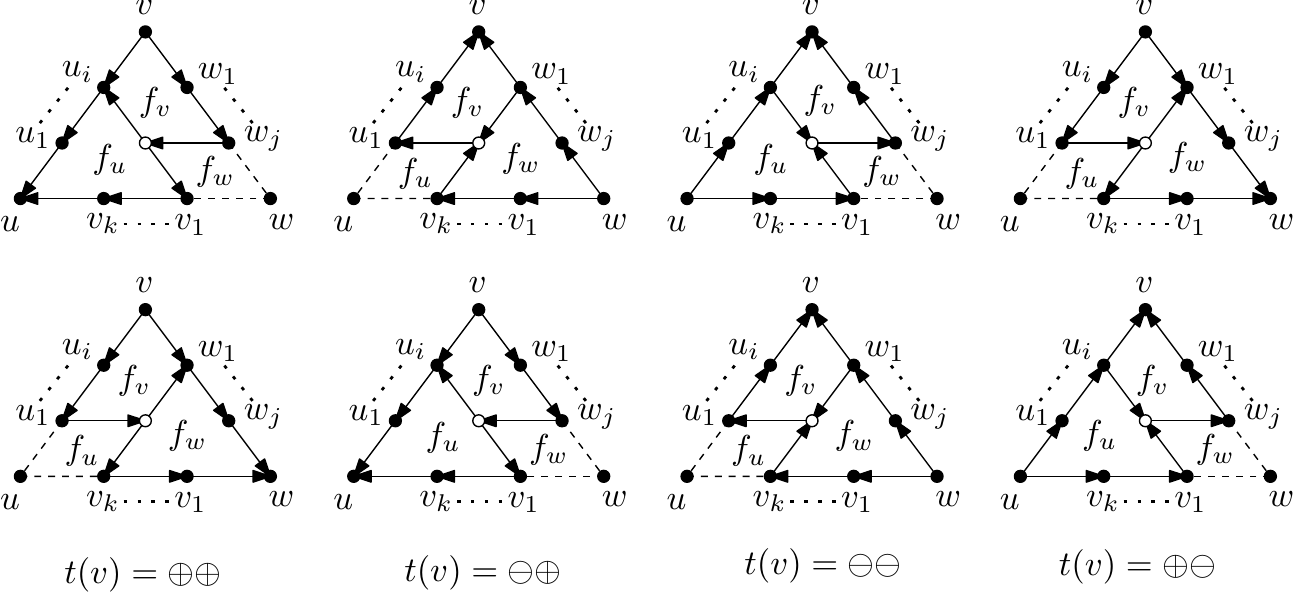}
 \caption{The directed graphs $D_r$ (top row) and $D_b$ (bottom row) locally around an inner face $f$ of $G$. The directions of dashed edges depend on the types of the endpoints.}
 \label{fig:ineq-around-face2}
\end{figure}

\begin{lemma}\label{lem:source-or-sink}
  The two edges incident to the blue (red) sink $s$ of $f$ are directed in $D_r$ ($D_b$) both incoming at $s$ if $t_r(s) = \ominus$ ($t_b(s) = \ominus$) and both outgoing at $s$ if $t_r(s) = \oplus$ ($t_b(s) = \oplus$).
\end{lemma}
\begin{proof}
 If both edges $e$ and $e'$ are incoming at the blue (red) sink $s$ in the edge labeling, then they are colored blue (red). By definition $e$ and $e'$ are incoming at $s$ in $D_r$ ($D_b$) if and only if the red (blue) sign of $s$ is $\ominus$.

 If in the edge labeling $e$ is incoming at $s$ while $e'$ is outgoing, then $e$ is blue (red) and $e'$ is red (blue). (Note that by the face rule $e$ and $e'$ can not be both outgoing at a sink of that face.) Again $e$ and $e'$ are incoming at $s$ in $D_r$ ($D_b$) iff the red (blue) sign of the end-vertex of $e$, which is $s$, and the start-vertex of $e'$, which is $s$ as well, is $\ominus$.
\end{proof}
%
%\begin{figure}[htb]
% \centering
% \includegraphics[width = \textwidth]{ineq-around-face}
% \caption{The directed graphs $D_b$ (top row) and $D_r$ (bottom row) locally around an inner face $f$ of $G$. The two edges that are not listed in Table~\ref{tab:inequalities-around-face} are drawn dashed.}
% \label{fig:ineq-around-face}
%\end{figure}
%

We need to ensure that the L-contact representation is non-crossing. The inequalities above are not sufficient to achieve this. Therefore we add additional inequalities for each inner face. These inequalities ensure that each inner face does not cross itself in the L-contact representation. The inequalities for each type of face are shown in Table~\ref{tab:face-edges}. The corresponding edges in $D_r$ and $D_b$, i.e., those that join a vertex with an inner face of $G$, are defined as follows. Consider an inner face $f$ with its three distinguished vertices $u,v,w$. If $t(v) = \I$ then we have in $D_r$ the edges $w_j \to f, f \to v_1$ and $f \to u_i$, and in $D_b$ the edges $u_1 \to f, f \to v_k$ and $f \to w_1$. If $t(v) = \II$ then we have in $D_r$ the edges $v_k \to f, f \to u_1$ and $w_1 \to f$, and in $D_b$ the edges $w_j \to f, f \to v_1$ and $f \to u_i$. If $t(v) = \III$ then we have in $D_r$ the edges $u_i \to f, v_1 \to f$ and $f \to w_j$, and in $D_b$ the edges $w_1 \to f, v_k \to f$ and $f \to u_1$. If $t(v) = \IV$ then we have in $D_r$ the edges $u_1 \to f, f \to v_k$ and $f \to w_1$, and in $D_b$ the edges $u_i \to f, v_1 \to f$ and $f \to w_j$. We again refer to Figure~\ref{fig:ineq-around-face2} for an illustration.

In case any of $u_1,u_i,v_1,v_k,w_1,w_j$ does not exist, we replace it in the above definition as follows: Replace $v_1/u_i$ by $u$, $u_1/w_j$ by $v$, and $v_k/w_1$ by $w$. This may introduce parallel edges, e.g., when $t(v) = \I$ and neither $v_1$ nor $u_i$ exists.
\begin{table}[h]
 \centering
 \newcommand{\la}{\ensuremath{\leftarrow}}
 \vspace{-1\baselineskip}
 \begin{tabular}{|c|c|c|c|c|}
  \hline
  & $t(v) = \I$ & $t(v) = \II$ & $t(v) = \III$ & $t(v) = \IV$ \\
  \hline
  $D_r$ & $w_j \to f$;$f \to v_1, u_i$ & $v_k,w_1 \to f$;$f \to u_1$ & $v_1,u_i \to f$;$f \to w_j$ & $u_1 \to f$;$f \to w_1, v_k$ \\
  \hline
  $D_b$ & $u_1 \to f$;$f \to w_1, v_k$ & $w_j \to f$;$f \to v_1, u_i$ & $v_k,w_1 \to f$;$f \to u_1$ & $v_1,u_i \to f$;$f \to w_j$ \\
  \hline
 \end{tabular}

 \medskip

 \caption{The three inequality edges of a face $f$ of $G$ in $D_r$ and $D_b$ for each type of $f$.}
 \vspace{-2\baselineskip}
 \label{tab:face-edges}
\end{table}

\begin{lemma}\label{lem:ineq_acyclic}
 The graphs $D_r$ and $D_b$ are acyclic.
\end{lemma}
\begin{proof}
 First note that $D_r$ and $D_b$ are planar. More precisely, either graph inherits a plane embedding from $G$ by putting a vertex for each inner face $f$ into the corresponding bounded region and connecting it by three edges to some of its incident vertices. This way $f$ is divided into three inner faces $f_u, f_v, f_w$, each corresponding to a different distinguished vertex of $f$, that is $u,v,w$ are incident to $f_u,f_v,f_w$, respectively. See Figure~\ref{fig:ineq-around-face2} for an illustration.

 To prove that $D_r$ is acyclic, it now suffices to show that every inner face of $D_r$ is acyclic, the special vertex $v_2$ is the only vertex with only outgoing edges and $v_1$ is the only vertex with only incoming edges. Similarly we want to show that every inner face of $D_b$ is acyclic, $v_1$ is the unique source in $D_b$, and $v_2$ the unique sink.

 Let us consider only $D_r$, since an analogous argumentation holds for $D_b$. Every inner face of $D_r$ is one of the three faces $f_u,f_v,f_w$ that correspond to an inner face $f$ of $G$. We consider $f$, its three distinguished vertices $u,v,w$ and assume w.l.o.g. that $t(v) = \I$. The cases that $t(v) \in \{\II,\III,\IV\}$ are similar. We want to show that each of $f_u,f_v,f_w$ is acyclic, i.e., contains a vertex whose two incident edges in that face are either both incoming or both outgoing:

 \begin{enumerate}[label=(\roman*)]
 \item The face $f_u$ contains the edges $f \to u_i$ (or $f \to u$) and $f \to v_1$ (or $f \to u$)\footnote{If neither $u_i$ nor $v_1$ exists, then $f_u$ consists only of two parallel edges directed from $f$ to $u$.} In particular, both edges at the vertex $f$ are outgoing and thus $f_u$ is acyclic.

 \item The face $f_w$ is a quadrangle consisting of the vertices $w$, $v_1$ (or $u$), $f$, and $w_j$ (or $v$). The two edges incident to $w$ are its two edges in the face $f$ of $G$. Thus by Lemma~\ref{lem:source-or-sink} $f_w$ is acyclic.

 \item The face $f_v$ contains the edges $f \to u_i$ and $v \to u_i$ if $u_i$ exists, and the edges $f \to u$ and $v \to u$ if $u_i$ does not exist. Thus both edges at $u_i$ or $u$ are incoming, and $f_v$ is acyclic.
 \end{enumerate}

 \noindent
 It remains to show that every vertex different from $v_1,v_2$ has at least one incoming and one outgoing edge in $D_r$. This is true by definition for vertices that correspond to inner faces of $G$. For every inner vertex $v$ of $G$ consider the inner face $f$ of $G$ such that $(v,f)$ is in the angular structure but not in the angular matching $M$. This means that $v$ is in the set $\{u_1,\ldots,u_i,v_1,\ldots,v_k,w_1,\ldots,w_j\}$ with respect to the face $f$. Now Table~\ref{tab:inequalities-around-face} and the definition of the three edges in $D_r$ incident to $f$ imply that $v$ has one incoming and one outgoing edge locally around the face $f$. Finally, by definition we have $v_2 \to v_3$ and $v_3 \to v_1$ in $D_r$, which concludes the proof.
\end{proof}
Recall that, by the edge rule of the edge labeling (see Section~\ref{sec:edge-labeling}), every non-special edge $u \to v$ is associated with one of its incident faces, such that $v$ is a sink of this face. Let $B_1,B_2,R_1,R_2$ be the four blocks of incoming blue and red edges at $v$, where $B_1$, $e_r(v)$, $B_2$, $R_1$, $e_b(v)$, $R_2$ appear around $v$ in this clockwise circular order. As illustrated in Figure~\ref{fig:Labeling}(d), each face within $R_1$ and $B_1$ is associated with the counterclockwise next incident edge and each face within $R_2$ and $B_2$ with the clockwise next incident edge.

\begin{lemma}\label{lem:ineq-around-vertex}
 Let $u \to v$ be a blue (red) edge and $f$ be the face associated with it. Then $D_r$ ($D_b$) contains the edge $u \to f$ if $t_r(v) = \oplus$ ($t_b(v) = \oplus$) and the edge $f \to u$ if $t_r(v) = \ominus$ ($t_b(v) = \ominus$).
\end{lemma}
\begin{proof}
We prove the statement only for a blue edge $u \to v$, i.e., $v$ is the blue sink of the face $f$. The argument for red edges is analogous.

Consider the red sign of $v$ and the blue sign of $u$. From the type rule follows that $t_r(v) \neq t_b(u)$ if $u \to v \in B_1$ and $t_r(v) = t_b(u)$ if $u \to v \in B_2$. Indeed, if $v$ is odd, i.e., $t_r(v) = t_b(v)$, then $u \to v$ lies in $B_1$ if and only if $u \to v$ lies in the matched angle of $v$, which is the case if and only if $t_b(u) \neq t_b(v) = t_r(v)$. Similarly, if $v$ is even, i.e., $t_r(v) \neq t_b(v)$, then $u \to v$ lies in $B_1$ if and only if $u \to v$ lies in the \emph{un}matched angle of $v$, which is the case if and only if $t_b(u) = t_b(v) \neq t_r(v)$.

If $w$ denotes the vertex that the face $f$ is matched to, then from Lemma~\ref{lem:types-around-face} follows $t_b(w) = t_b(u)$. If $t_b(w) = \oplus$ (and hence $t_r(v) = \ominus$), then the edges in $D_r$ between vertices and faces of $G$ are directed counterclockwise around the blue sink of $f$ (which is $v$). Now by the edge rule of the edge labeling (see Section~\ref{sec:edge-labeling}), $u$ comes counterclockwise before $v$ on $f$ if and only if $u \to v$ lies in $B_2$. Thus we have the edge $u \to f$ in $D_r$ iff $u \to v \in B_2$, which is the case iff $t_r(v) = t_b(u) = \ominus$.

Figure~\ref{fig:ineq-around-vertex} shows how some edges around $v$ are directed in $D_b$.
\begin{figure}[htb]
\centering
\includegraphics{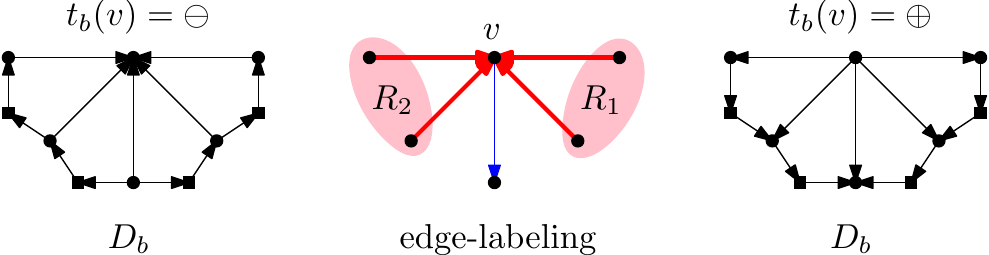}
\caption{The inequality graph $D_b$ around a vertex $v$.}
\label{fig:ineq-around-vertex}
\end{figure}
\end{proof}

%%%%%%%%%%%%%%%%%%%%%%%%%%%%%%%%%%%%%%%
%%%%%%% C O N S T R U C T I O N %%%%%%%
%%%%%%%%%%%%%%%%%%%%%%%%%%%%%%%%%%%%%%%
\subsection{Construction}\label{sec:construction}

%In this section we finally construct an L-contact representation of $G$ using all the combinatorial structures derived in the preceeding sections.

Given a planar Laman graph $G$, an L-contact representation of $G$ is constructed as follows:
\begin{enumerate}[label = (\arabic*)]
 \item Find a planar Henneberg construction for $G$.\label{enum:find-Henneberg}
 \item Compute an angular tree $T$ of $G$ (Theorem~\ref{thm:tree-if-Laman}).\label{enum:angular-structure}
 \item Compute the angle and edge labeling of $G$ w.r.t. $T$ (Theorem~\ref{thm:angles-from-structure} and~\ref{thm:edge-from-angle}).\label{enum:angle-edge-labeling}
 \item Compute the type of every vertex of $G$ according to the type rule in Section~\ref{sec:types}. This can be computed using a simple traversal of the trees $E_r$ and $E_b$.\label{enum:types}
 \item Define the directed graphs $D_r$ and $D_b$ as described in Section~\ref{sec:inequalities}\label{enum:ineq-graphs}
 \item Compute a topological order of $D_r$ and $D_b$ and let, for every vertex $v$ in $G$, $x(v)$ and $y(v)$ be the number of $v$ in these topological orders, respectively.\label{enum:topological-order}
 \item For every non-special vertex $v$ with $v \to u$ in $E_r$ and $v \to w$ in $E_b$ define an L-shape $\LL(v)$ whose horizontal leg spans from $x(v)$ to $x(u)$ on $y$-coordinate $y(v)$ and whose vertical leg spans from $y(v)$ to $y(w)$ on $x$-coordinate $x(v)$.\label{enum:define-L}
\end{enumerate}
Let $n$ be the number of vertices of $G$. By Theorem~\ref{thm:tree-if-Laman} we can compute an angular tree of $G$ in $\mathcal{O}(n^2)$ time. The angle labeling w.r.t. $T$ can be computed in $\mathcal{O}(n)$ time using the linear time algorithm of de Fraysseix and Ossona de Mendez~\cite{deFraysseix200157}. Similarly, the edge labeling w.r.t. $T$ can be computed by a simple traversal of the tree $H$ described in the proof of Theorem~\ref{thm:edge-from-angle}. It is easy to see that the remaining steps of our algorithm can also be computed in $\mathcal{O}(n)$ time. Finally note that the vertices of $D_r$ and $D_b$ that correspond to inner faces of $G$ do not need to be included in the topological order of $D_r$ and $D_b$. Hence every coordinate used in the L-contact representation is between $1$ and $n$.

Before we prove the correctness and runtime of the above algorithm let us refer to Figure~\ref{fig:BigExample} for a detailed example showing all the important structures that are computed during the algorithm.

\begin{figure}[h!]
\centering
\includegraphics[width=.95\textwidth]{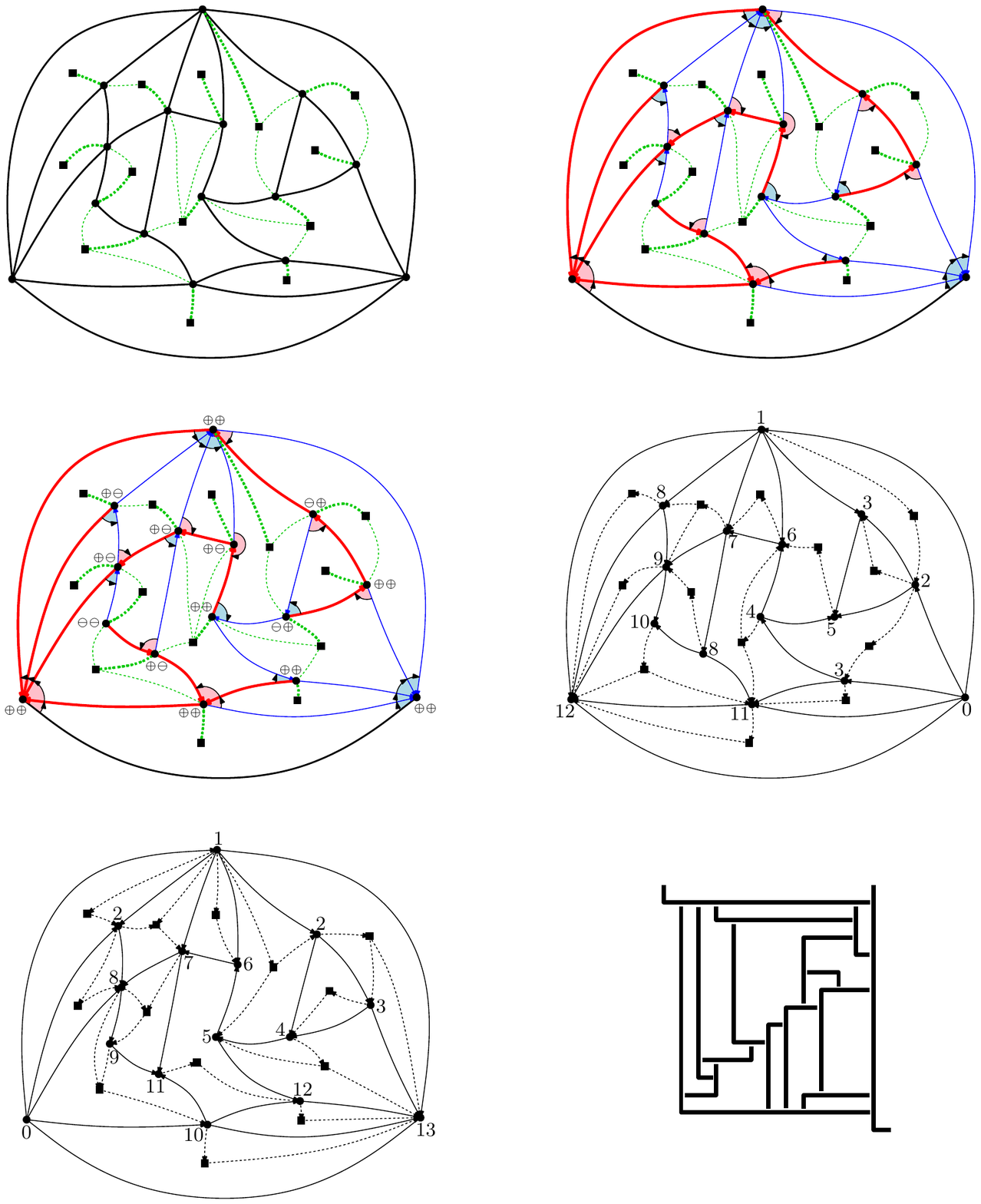}
\caption{Left top: angular tree, and the corresponding matching (thick). Right top: edge labeling corresponding to angular tree. Middle left: vertex types. Middle right: inequality graph $D_r$ plus $x$-coordinates. Bottom left: inequality graph $D_b$ plus $y$-coordinates. Bottom right: L-contact representation.}
\label{fig:BigExample}
\end{figure}

\begin{theorem}\label{thm:LContactRep}
The algorithm above computes an L-contact representation of $G$ on an $n \times n$ grid in $\mathcal{O}(n^2)$ time, where $n$ is the number of vertices of $G$. If an angular tree is given, then the algorithm runs in $\mathcal{O}(n)$~time.
\end{theorem}
\begin{proof}
 The running time of the algorithm and the size of the drawing have already been argued in Section~\ref{sec:construction}. It remains to show that the constructed L-shapes indeed form an L-contact representation of $G$. First note that an L-shape $\LL(v)$ is of type $t$ if and only if the corresponding vertex has type $t$. To see this, consider a non-special vertex $v$ of $G$ with $t_r(v) = \oplus$. Let $u$ be its outgoing red neighbor in the edge labeling, that is, in $E_r$ we have the edge $v \to u$. By definition (case (i)) $D_r$ contains the edge $v \to u$ and thus $x(v) < x(u)$. This means that the vertical leg of $\LL(u)$ lies to the right of the vertical leg of $\LL(v)$. In particular $\LL(v)$ is either of type I or IV, as desired. Similarly assume that $t_b(v) = \ominus$ and consider the outgoing blue neighbor $w$ of $v$. Then by definition (case (ii)) $D_b$ contains the edge $w \to v$, which implies that $\LL(v)$ has type III or IV. Thus if $t(v) = \IV$ then $\LL(v)$ is of type IV. The cases $t(v) \in \{\I,\II,\III\}$ are similar.

\smallskip\noindent
 The rest of the proof is divided into several claims.

 \medskip
 \noindent
 \textit{Claim~1. All edges in $G$ are represented by non-degenerate point contacts of the corresponding L-shapes.}

 \smallskip\noindent
 \textit{Proof.} W.l.o.g. consider a red edge $u \to v \in E_r$. The horizontal endpoint of $\LL(u)$ is given by $(x(v),y(u))$ and the vertical leg of $\LL(v)$ is supported by the line $x = x(v)$, but it remains to show that $(x(v),y(u))$ lies on the vertical leg of $\LL(v)$, i.e., $y(v) < y(u) < y(w)$ or $y(v) > y(u) > y(w)$, where $w$ is the outgoing blue neighbor of $v$.

  If $t_b(v) = \oplus$, then by Lemma~\ref{lem:ineq-around-vertex} there is a directed path in $D_b$ from $v$ via $u$ to the outgoing blue neighbor $w$ of $v$ and hence $y(v) < y(u) < y(w)$. Similarly if $t_b(v) = \ominus$, then there is a directed path in $D_b$ from $w$ via $u$ to $v$ and hence $y(v) > y(u) > y(w)$, which is what we wanted to show.
 \hfill $\triangle$ \textit{Claim~1.}

 \medskip
 \noindent
 \textit{Claim~2. Going around $\LL(v)$ the contacts with other L-shapes appear in the same cyclic order as the incident edges of $v$ in the plane embedding of $G$.}

 \smallskip\noindent
 \textit{Proof.} Consider the vertex $v$ in the edge labeling and assume w.l.o.g. that $t(v) = \I$, i.e., $\LL(v)$ is of type I. Recall that $R_1,R_2,B_1,B_2$ denote the blocks of incoming red and blue edges around $v$ in this clockwise cyclic order. By the type rule we know that $t_b(u) = \ominus$ for each $u \to v \in B_1$, $t_b(u) = \oplus$ for each $u \to v \in B_2$, $t_r(u) = \oplus$ for each $u \to v \in R_1$, and $t_r(u) = \ominus$ for each $u \to v \in R_2$. Since the types of L-shapes match the types of the vertices they represent, we get that each $\LL(u)$ makes contact with $\LL(v)$ on the correct side of the correct leg of $\LL(v)$, e.g., $\LL(u)$ touches the vertical leg of $\LL(v)$ from the left for $u \to v \in R_1$ and so on.

 It remains to show that within each block the contacts appear in the same cyclic order as the corresponding edges in the plane embedding of $G$. Consider any block, say $R_1$, and still assume that $t_b(v) = \oplus$. Let $w$ be the outgoing blue neighbor of $v$. Then by Lemma~\ref{lem:ineq-around-vertex} there is a directed path in $D_b$ starting at $v$, going through all incoming red neighbors in $R_1$ in clockwise order, and ending at $w$ (see Figure~\ref{fig:ineq-around-vertex}). In other words, the y-coordinates of $\LL(v)$, the L-shapes in $R_1$ in clockwise order, and $\LL(w)$ are increasing, which is what we wanted to show.

 The consideration of $R_2,B_1$ and $B_2$, as well as cases with $t_b(v) \neq \oplus$ are analogous.
 \hfill $\triangle$ \textit{Claim~2.}

 \medskip
 \noindent
 \textit{Claim~3. Every face $f$ of $G$ corresponds to a rectilinear polygonal region whose boundary is contained in the L-shapes corresponding to the vertices of $f$.}

 \smallskip\noindent
 \textit{Proof.}
 It is easy to see that the statement holds for the outer face. So consider any inner face $f$ and let $u,v,w$ be its three distinguished vertices. Let us trace the polygonal path $P$ that is the claimed boundary of the region corresponding to $f$. Start at the bend of $\LL(v)$ and go along the vertical leg on its right side if $v$ is odd and on its left side if $v$ is even. Whenever we meet a contact we turn right for $v$ odd and left for $v$ even, and traverse the other L-shape on the corresponding side. From Claim~2 follows that $P$ is a closed path (corresponding to the inner face $f$).

 \begin{figure}[htb]
  \centering
  \includegraphics{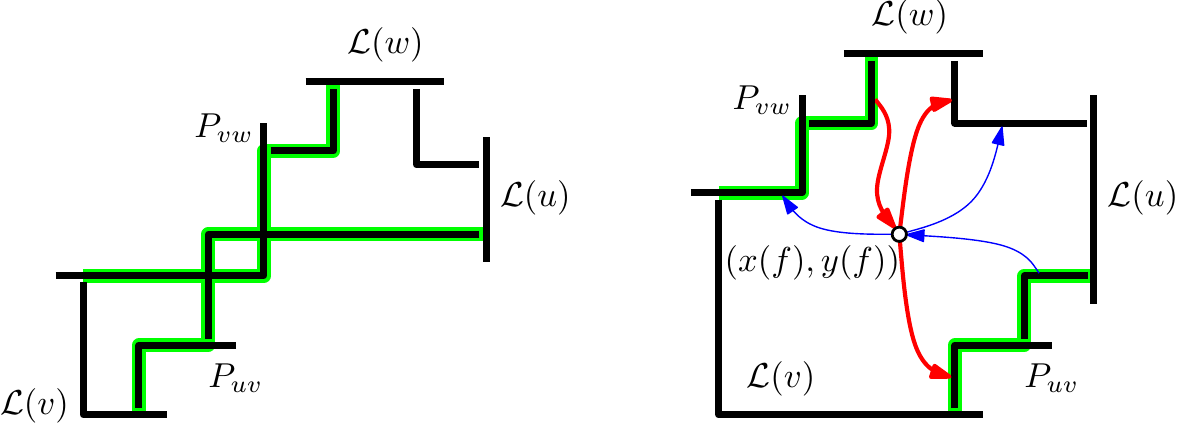}
  \caption{Left: A situation where $P_{vw}$ and $P_{uv}$ intersect. Right: The introduction of the point $(x(f),y(f))$ forces $P_{vw}$ and $P_{uv}$ to be disjoint.}
  \label{fig:face-region}
 \end{figure}

 The type of each vertex of $f$ (except for $u$ and $w$) is given by Lemma~\ref{lem:types-around-face}. Moreover, each such vertex has an edge with $f$ in the angular structure, which means that the bend of the corresponding L-shape is on $P$. Finally, from Table~\ref{tab:inequalities-around-face} and Figure~\ref{fig:ineq-around-face2} we see that $P$ is divided into three monotone parts, $P_{vw}$ between $\LL(v)$ and $\LL(w)$, $P_{wu}$ between $\LL(w)$ and $\LL(u)$, and $P_{uv}$ between $\LL(u)$ and $\LL(v)$. So the only thing that could happen is that $P_{vw}$ intersects $P_{uv}$ as illustrated in the left of Figure~\ref{fig:face-region}.

 Recall that we introduced a point $(x(f),y(f)) \in \mathbb{R}^2$ associated with the face $f$. We claim that this point ensures that $P$ is non-crossing. Consider for example the case that $\LL(v)$ has type I. If $P_{vw} \cap P_{uv} \neq \emptyset$ then $x(w_j) > x(u_i)$ and $y(u_1) > y(w_1)$. But by definition we have a path $w_j \to f \to u_i$ in $D_r$ and a path $u_1 \to f \to w_1$ in $D_b$ (see Figure~\ref{fig:face-region} right). Thus $x(w_j) < x(f) < x(u_i)$ and $y(u_1) < y(f) < y(w_1)$, which means that $P$ is indeed not self-intersecting and thus proves the claim.
 \hfill $\triangle$ \textit{Claim~3.}

 \medskip
 \noindent
 Next, we consider the embedding of $G$ inherited from the touching L-shapes (vertices are placed inside the corresponding L and edges are drawn along the L-shapes through the corresponding touching point.). By Claim~2 this embedding has the correct rotation scheme and by Claim~3 every face is crossing-free. Since $G$ is $2$-connected it follows that this embedding is the plane embedding of $G$ we started with. In particular no two L-shapes cross each other. This completes the proof.
%
%  \medskip
%  \noindent
%  Consider for each face the corresponding region of the plane, i.e., the rectilinear polygon for each inner face and the unbounded region for the outer face. From Claim~2 and Claim~3 follows that for each vertex $v$ the regions corresponding to faces at $v$ appear consecutively around $\LL(v)$ in the same cyclic order as in the plane embedding of $G$.
%
%  Next we derive an embedding of $G$ from the L-shapes as follows. We place every vertex $v$ onto $(x(v),y(v))$, i.e., the bend of $\LL(v)$. Then we thicken each L-shape by some tiny amount and draw each edge $u \to v$ as follows: Starting at $(x(u),y(u))$ and ending at $(x(v),y(v))$ we route a continuous curve within the union of the thickened $\LL(u)$ and the thickened $\LL(v)$, passing through their contact point. We make sure that the curves within each L-shape do not intersect, except in their endpoints. By Claim~2 this give at every vertex the rotation scheme we started with. Morevoer, by Claim~3 every facial cycle is drawn without crossings.
%
%  Hence we have defined a drawing of $G$ with a planar rotation scheme and with crossing-free facial cycles. Since $G$ is $2$-connected, this implies that our drawing is plane. In particular the only points in common of two L-shapes are the contact points that represent the edges of $G$. This completes the proof.
\end{proof}
%With Lemma~\ref{lem:L-and-Seg} and Theorem~\ref{thm:Laman-are-L-contact} we have now proven the following.
%
%\begin{theorem}
% The class of planar Laman graphs, 2-segment graphs and proper L-contact graphs are the same.
%\end{theorem}
%
%Moreover, we gave a procedure that computes a proper L-contact representation of every planar Laman graph in polynomial time. Except for step~\ref{enum:find-Henneberg} this procedure runs even in linear time, while the best-known runtime for finding a Henneberg sequence is \textbf{TU: Well, what is it?}. Let us remark that the L-contact representation uses only the $n \times n$ grid.

\section{Future Work and Open Problems}

Using our newly discovered combinatorial structure, we showed that planar Laman graphs are L-contact graphs. Thus, we showed that axis-aligned L's are as "powerful" as segments with arbitrary slopes when it comes to {\em contact representation of planar graphs}~\cite{A+11}. The equivalent result is not true for {\em intersection representation of planar graphs}. Indeed there is no $k$ such that all segment intersection graphs have an intersection representation with axis-aligned paths with no more than $k$ bends each~\cite{ChaplickJKV12}.

We think that L-contact representations can be used in various settings. For example, by ``fattening`` the L's we can get proportional side-contact representations similar to those in~\cite{A+11}.

Several natural open problems follow from our results:
\begin{enumerate}
\item There are L-contact graphs that are not Laman graphs (e.g. $K_4$). All L-contact graphs are planar and satisfy $|E(W)| \leq 2|W| -2$ for all $W \subseteq V$. Are these conditions also sufficient?
\item The L-contact representations resulting from our algorithm use all four types of L-shapes. If we limit ourselves to only type-I L's we can represent planar graphs of tree-width at most $2$, which include outerplanar graphs. What happens if we limit ourselves to only type-I L's {\bf and} allow degenerate L's?
\item Not every edge labeling corresponds to an angular tree. What are the necessary conditions for an edge labeling to have a corresponding (not necessarily proper) L-contact representation?
\item Planar Laman graphs can be characterized by the existence of an angular tree, which we can compute in $\mathcal{O}(n^2)$ time. This is slower than the fastest known algorithm for recognizing Laman graphs, which runs in $\mathcal{O}(n^{3/2}\sqrt{\log n})$ time~\cite{DaescuK09}. Can we compute angular trees faster, as to obtain a faster algorithm for recognizing planar Laman graphs?
\end{enumerate}

\smallskip
\noindent{\bfseries Acknowledgments.} The research in this paper started during the \emph{Bertinoro Workshop on Graph Drawing}. The authors gratefully acknowledge the other participants for useful discussions.

\newpage

\bibliographystyle{abbrv}
{
\bibliography{laman}
}

\end{document}